\documentclass{amsart}[10pt]
\usepackage[utf8]{inputenc}

 \usepackage[margin=1in,headheight=13.6pt]{geometry}

\usepackage{textcmds} 
\usepackage{amsmath}
\usepackage{amsxtra}
\usepackage{amscd}
\usepackage{amsthm}
\usepackage{amsfonts}
\usepackage{amssymb}
\usepackage{eucal}
\usepackage[all]{xy}
\usepackage{graphicx}
\usepackage{comment}
\usepackage{epsfig}
\usepackage{psfrag}
\usepackage{mathrsfs}
\usepackage{amscd}
\usepackage{rotating}
\usepackage{lscape}
\usepackage{amsbsy}
\usepackage{verbatim}
\usepackage{moreverb}
\usepackage{url}
\usepackage{extarrows}
\usepackage{setspace}
\usepackage{todonotes}

\usepackage{cancel}
\usepackage{bbm}

\usepackage{xcolor}
\usepackage{hyperref}
\usepackage{cleveref}

\newcommand\myshade{85}
\colorlet{mylinkcolor}{violet}
\colorlet{mycitecolor}{blue}
\colorlet{myurlcolor}{green}

\hypersetup{
  linkcolor  = mylinkcolor!\myshade!black,
  citecolor  = mycitecolor!\myshade!black,
  urlcolor   = myurlcolor!\myshade!black,
  colorlinks = true,
}

\usepackage{tikz-cd}
\usepackage{tikz}
\usetikzlibrary{matrix,arrows,decorations.pathmorphing}

\usepackage{pictexwd,dcpic}

\renewcommand{\sec}[0]{\section}
\newcommand{\ssec}[0]{\subsection}
\newcommand{\sssec}[0]{\subsubsection}

\theoremstyle{plain}
\newtheorem{thm}[subsubsection]{Theorem}

\newtheorem{mainthm}{Theorem}

\newtheorem{lem}[subsubsection]{Lemma}
\newtheorem{prop}[subsubsection]{Proposition}
\newtheorem{cor}[subsubsection]{Corollary}
\newtheorem{conj}[subsubsection]{Conjecture}
\theoremstyle{definition}

\newtheorem{defn}[subsubsection]{Definition}

\theoremstyle{plain}

\theoremstyle{definition}

\theoremstyle{remark}
\newtheorem{rmk}[subsubsection]{Remark}

\newtheorem{rem}[subsubsection]{Remark}

\newcommand{\on}{\operatorname}
\newcommand{\nc}{\newcommand}
\newcommand{\renc}{\renewcommand}

\newcommand{\ul}[1]{\underline{#1}}

\newcommand{\op}[0]{{\on{op}}}

\newcommand{\oo}[0]{\infty}

\newcommand{\uH}[0]{\on H}
\newcommand{\uG}[0]{\on G}
\newcommand{\uI}[0]{\on{I}}
\newcommand{\uK}[0]{\on{K}}

\newcommand{\Tr}[0]{\on{Tr}}

\newcommand{\coFib}[0]{\on{coFib}}

\renewcommand{\rightarrow}{\to}
\newcommand{\xto}[1]{\xrightarrow{#1}}
\newcommand{\tto}{\twoheadrightarrow}
\newcommand{\hto}{\hookrightarrow}
\nc{\squigto}{\rightsquigarrow}

\newcommand{\ccC}[0]{\mathcal{C}}

\newcommand{\ccH}[0]{\mathcal{H}}

\newcommand{\ccM}[0]{\mathcal{M}}

\newcommand{\ccO}[0]{\mathcal{O}}

\newcommand{\ccS}[0]{\mathcal{S}}

\nc{\CA}{{\mathcal{A}}}
\nc{\CB}{{\mathcal{B}}}

\nc{\CE}{{\mathcal{E}}}
\nc{\CF}{{\mathcal{F}}}

\nc{\CL}{{\mathcal{L}}}
\nc{\CC}{{\mathcal{C}}}
\nc{\CG}{{\mathcal{G}}}
\nc{\CM}{{\mathcal{M}}}
\nc{\CN}{{\mathcal{N}}}
\nc{\CK}{{\mathcal{K}}}
\nc{\CO}{{\mathcal{O}}}
\nc{\CP}{{\mathcal{P}}}
\nc{\CQ}{{\mathcal{Q}}}
\nc{\CR}{{\mathcal{R}}}
\nc{\CS}{{\mathcal{S}}}
\nc{\CU}{{\mathcal{U}}}
\nc{\CV}{{\mathcal{V}}}
\nc{\CW}{{\mathcal{W}}}
\nc{\CX}{{\mathcal{X}}}
\nc{\CY}{{\mathcal{Y}}}
\nc{\CZ}{{\mathcal{Z}}}
\nc{\CI}{{\mathcal{I}}}
\nc{\I}{\CI}

\newcommand{\bbA}[0]{\mathbb{A}}

\newcommand{\bbC}[0]{\mathbb{C}}
\newcommand{\bbD}[0]{\mathbb{D}}

\newcommand{\bbG}[0]{\mathbb{G}}

\newcommand{\bbQ}[0]{\mathbb{Q}}

\newcommand{\bbU}[0]{\mathbb{U}}

\newcommand{\bbZ}[0]{\mathbb{Z}}

\newcommand{\ffF}[0]{\mathfrak{F}}

\newcommand{\rmB}[0]{\mathrm{B}}

\newcommand{\scrV}[0]{\mathscr{V}}

\newcommand{\Ql}[1]{\bbQ_{\ell,#1}}
\newcommand{\BU}[0]{\rmB \bbU}

\newcommand{\Spec}{\on{Spec}}

\newcommand{\pr}{\on{pr}}

\newcommand{\Tor}[0]{\on{Tor}}
\newcommand{\Ext}[0]{\on{Ext}}

\newcommand{\HH}[0]{\on{HH}}
\newcommand{\HK}[0]{\on{HK}}
\newcommand{\dgCAT}[0]{\on{dgCAT}}
\newcommand{\Mod}[0]{\on{Mod}}

\newcommand{\SH}[0]{\ccS \ccH}

\newcommand{\Mv}[0]{\ccM^{\vee}}
\newcommand{\rl}[0]{\on{r}^{\ell}}
\newcommand{\chern}{\ccC\!\on{h}^{\ell}}

\newcommand{\Sw}[0]{\on{Sw}}
\newcommand{\dimtot}[0]{\on{dimtot}}

\newcommand{\Qell}{\bbQ_{\ell}}

\newcommand{\QellI}{\bbQ_\ell^{\uI}}

\newcommand{\rell}{\rl}

\renewcommand{\epsilon}{\varepsilon}
\newcommand{\eps}{\epsilon}

\newcommand{\dgCat}{\on{dgCat}}
\newcommand{\longto}{\longrightarrow}

\nc{\rev}{{\on{rev}}}
\nc{\env}{{\on{env}}}

\nc{\sA}{{\mathsf{A}}}
\nc{\sB}{{\mathsf{B}}}
\nc{\sC}{{\mathsf{C}}}
\nc{\sD}{{\mathsf{D}}}
\nc{\sE}{{\mathsf{E}}}
\nc{\sF}{{\mathsf{F}}}
\nc{\sG}{{\mathsf{G}}}
\nc{\sK}{{\mathsf{K}}}
\nc{\sM}{{\mathsf{M}}}
\nc{\sN}{{\mathsf{N}}}
\nc{\sO}{{\mathsf{O}}}
\nc{\sW}{{\mathsf{W}}}
\nc{\sQ}{{\mathsf{Q}}}
\nc{\sP}{{\mathsf{P}}}
\nc{\sR}{{\mathsf{R}}}
\nc{\sS}{{\mathsf{S}}}
\nc{\sT}{{\mathsf{T}}}
\nc{\sU}{{\mathsf{U}}}
\nc{\sV}{{\mathsf{V}}}
\nc{\sZ}{{\mathsf{Z}}}

\DeclareMathOperator{\Coh}{\mathsf{D^b_{coh}}} %Coh
\DeclareMathOperator{\Perf}{\mathsf{D_{perf}}} %Perf
\DeclareMathOperator{\QCoh}{\mathsf{D_{qcoh}}} %QCoh
\DeclareMathOperator{\Sing}{\mathsf{D_{sg}}} %Singularity  cat

\DeclareMathOperator{\Hom}{Hom} %Hom space
\DeclareMathOperator{\ev}{ev} %evaluation
\DeclareMathOperator{\coev}{coev} %coevaluation
\DeclareMathOperator{\id}{id} %identity
\DeclareMathOperator{\MF}{MF} %matrix factorizations
\DeclareMathOperator{\cat}{cat} %categorical
\DeclareMathOperator{\In}{I} %inertia. 

\DeclareMathOperator*{\usotimes}{\otimes} %%per avere cose sotto il tensore

\nc{\virg}[1]{``#1"}

\nc{\bigt}[1]{\big( #1 \big) }
\nc{\Bigt}[1]{\Big( #1 \Big) }

\nc{\TV}{To\"en--Vezzosi}

\nc{\Bl}{\on{Bl}} %%% Bloch number, or Bloch class.
\nc{\Art}{\on{Ar}} %%% Artin number, or Artin class.

\nc{\Benv}{\CB^{\env}}

\nc{\sBenv}{\sB^{\env}}
\nc{\sCenv}{\sC^{\env}}

\nc{\et}{{\on{\acute{e}t}}}
\nc{\bareta}{\bar{\eta}}

\nc{\wh}{\widehat}
\nc{\T}{\sT}
\nc{\B}{\sB}

\nc{\HQ}{\on{H}\!\bbQ}
\nc{\shvl}{\textup{Shv}_{\bbQ_{\ell}}}

\nc\wt{\widetilde}

\nc{\lr}{\xymatrix{ \ar@<-0.4ex>[r] \ar@<.5ex>[l]  & } }
\nc{\rr}{\xymatrix{ \ar@<-0.2ex>[r] \ar@<.7ex>[r]  & } }
\nc{\rrr}{\xymatrix{ \ar@<.0ex>[r] \ar@<.7ex>[r] \ar@<-0.7ex>[r] & } }

\renc{\O}{\ccO}
\renc{\sim}{\simeq}
\nc{\heart}{\heartsuit}

\nc{\rlS}{\rl_{S}}
\nc{\Rl}{\CR^{\ell}}
\nc{\CH}{{\mathcal{H}}}

\nc{\QellSbeta}{\bbQ_{\ell,S}(\beta)}
\nc{\QellSIbeta}{\bbQ^{\In}_{\ell,S}(\beta)}

\nc{\restr}[2]{\left. #1 \right |_{#2}}

\nc{\Gm}{\bbG_m}

\nc{\green}[1]{\textcolor{green}{#1}} 
\nc{\red}[1]{\textcolor{red}{#1}}
\nc{\orange}[1]{\textcolor{orange}{#1}}

\author{Dario Beraldo \and Massimo Pippi}

\address[D.~Beraldo]{Department of Mathematics, University College London, London WC1H 0AY, United Kingdom}

\email{d.beraldo@ucl.ac.uk}

\address[M.~ Pippi]{Univ Angers, CNRS-UMR 6093, LAREMA, SFR MATHSTIC, F-49000 Angers, France}

\email{massimo.pippi@univ-angers.fr} 

\begin{document}

\title{Non-commutative intersection theory and unipotent Deligne--Milnor formula}

\begin{abstract}
We apply methods of derived and non-commutative algebraic geometry to understand intersection theoretic phenomena on arithmetic schemes.
Specifically, we categorify Bloch's intersection number (in the formulation provided by Kato--Saito).
Combining this with {\TV} non-commutative Chern character, we obtain a generalization of Bloch conductor conjecture in several new cases, including the unipotent Deligne--Milnor formula.

\end{abstract}

\maketitle
\tableofcontents

\sec{Introduction}

This paper is a contribution to \emph{intersection theory on arithmetic schemes} by means of derived and non-commutative algebraic geometry, a program envisioned by \TV{} (\cite{tv22}). 

Specifically, we categorify the localized intersection product introduced by Kato--Saito (\cite{ks04})
and, as an application, we prove the unipotent case of the \emph{Deligne--Milnor conjecture}, see \cite[Expos\'e XVI]{sga7ii}.
With the same method, we also prove several new cases of the \emph{unipotent Bloch conductor conjecture}.

\medskip

The results of the present paper, combined with those of \cite{bp24} where we provide a non-commutative interpretation of the \emph{total dimension} of arithmetic schemes, will settle the Deligne--Milnor conjecture in full generality (that is, without the unipotence assumption).

\ssec{The Deligne--Milnor conjecture} \label{ssec: intro DM-conj}

\sssec{}

Let $f: \bbC^{n+1} \to \bbC$ be an analytic function with an isolated critical point $x \in \bbC^{n+1}$ lying in the special fiber $X_0 := f^{-1}(0)$.
The celebrated Milnor formula states that the dimension of the Jacobian ring of $f$ at $x$ (also called \emph{Milnor number}) equals the number of vanishing cycles, see \cite{mi69}.

\sssec{}

In \cite[Expos\'e XVI]{sga7ii}, Deligne formulated a (conjectural) algebro-geometric version of this formula. 
In this situation, the map $f$ (or rather, its germ near the preimage of $0 \in \bbC$) is replaced by a map of schemes $p: U \to S$, where:

\begin{itemize}

\item
the base $S$ is a strictly henselian trait\footnote{It is not necessary to assume that the trait is strict. 
However, it turns out that there is no loss of generality in doing so.}.
For concreteness, the reader could consider $S = \Spec(\bbZ_p^{\textup{sh}})$, the spectrum of the strict henselization of the ring of $p$-adic integers, or $S = \Spec(k [\![ t]\!])$ for some separably closed field $k$ of arbitrary characteristic. 
Denote by $s$ the closed (or special) point of $S$, by $\eta$ the generic point and by $\bareta$ the geometric generic point.

\item
the total space $U$ is regular, while $p$ is a flat morphism that is smooth everywhere except for a closed point $x$ in the special fiber $U_s := U \times_S s$. 
Moreover, we assume that $U$ is purely of relative dimension $n$.

\end{itemize}

\sssec{}

Denote by $\Omega^1_{U/S}$ the coherent sheaf of relative Kahler differentials. 
In this situation, the Deligne--Milnor number is defined by 
$$
\mu_{U/S}
:=
\on{length}_{\CO_{U,x}}
\bigt{ 
\ul \Ext^1(\Omega^1_{U/S},\CO_U)_x
}.
$$
This is a generalization of the dimension of the Jacobian ring to the algebro-geometric situation and it will be the left hand-side of the Deligne--Milnor formula.

\sssec{}

We now discuss the right-hand side the Deligne--Milnor formula. 
Let $\ell$ be a prime number different from the residue characteristics of $S$.

Denote by $U_\eta$ and $U_{\bareta}$ the generic and geometric generic fibers, respectively. We know that $U_{\bareta}$ (and therefore its cohomology) carries a natural action of the \emph{inertia group} $\In$, which coincides with the absolute Galois group of the generic point of $S$ (since we are assuming that $S$ is strict). 

\sssec{}

Next, recall the sheaf of vanishing cycles of $U/S$: this is an $\ell$-adic sheaf on the special fiber, supported on the singular locus of $U/S$. 
For more details, see \cite{sga7i, sga7ii}.
In our case (where $U/S$ has only an isolated singularity), the sheaf of vanishing cycles can be identified with
a $\Qell$-vector space $\Phi_x$, 
placed in cohomological degree $n$,
equipped with a canonical action of the inertia group $\In$. 
The Swan conductor $\Sw(\Phi_x)$ is an integer related to the action of the wild inertia subgroup: see \cite[Expos\'e XVI]{sga7ii}, \cite{ab00} and \cite{ks04} for a precise definition. 
Following Deligne, we define the \emph{total dimension} of $\Phi_x$ by the formula
$$
\dimtot(\Phi_x)
= 
\dim(\Phi_x) + \Sw(\Phi_x).
$$

\begin{conj}[Deligne--Milnor formula, {\cite[Expos\'e XVI]{sga7ii}}] \label{conj:DM}
In the above situation, we have:
\begin{equation} \label{eqn:DM-intrp}
  \mu_{U/S}
  =
  (-1)^n\dimtot(\Phi_{x}).
\end{equation}
\end{conj}

\sssec{}

The following cases of the above conjecture where proven by Deligne in {\cite[Expos\'e XVI]{sga7ii}}:
\begin{enumerate}

\item
when $S$ has pure characteristic;

\item
when the relative dimension of $U$ over $S$ is zero;

\item
when the singularity at $x$ is ordinary quadratic.
\end{enumerate}

In \cite{or03}, Orgogozo showed that Conjecture \ref{conj:DM} is equivalent to a special case of \emph{Bloch conductor conjecture}, see below. 
In particular, by looking at the list of known cases of the latter conjecture, we see that Conjecture \ref{conj:DM} holds true when $n=1$.

Besides these cases, the conjecture remains open when $S$ has mixed characteristic.
We can now state the first main theorem of this paper.

\begin{mainthm} \label{mainthm:DM-unip}
The Deligne--Milnor conjecture holds true as soon as the inertia group acts unipotently on $\Phi_x$.
\end{mainthm}

\begin{rmk}
As mentioned, the present paper provides a non-commutative interpretation of the left-hand side of the Deligne--Milnor formula \eqref{eqn:DM-intrp}.

In the companion paper \cite{bp24}, we will interpret the right-hand side from this perspective as well.
The results of these two papers, combined, will allow us to settle the whole conjecture.
\end{rmk}

\ssec{Bloch conductor conjecture}

We now recall the statement of Bloch conductor conjecture (BCC) and its connection with the Deligne--Milnor conjecture.

\sssec{}\label{sssec:BCC-notation}

Consider an $S$-scheme $p:X \to S$ which is regular, flat, proper and generically smooth.
Notice that $p_s:X_s \to s$ might very well be singular. Here the singular locus can be arbitrary, we no longer require it to be a point.
Bloch conductor conjecture describes the difference of the $\ell$-adic Euler characteristics of $X_s$ and $X_{\bareta}$ as follows:

\begin{conj}[Bloch conductor formula, \cite{bl87}] \label{conj:Bloch}
For $p:X \to S$ as above, we have
\begin{equation} \label{eqn:Bloch-formula}
\chi(X_s;\Qell) -\chi(X_{\bareta};\Qell)
= \Bl(X/S) + \Sw(X_{\eta}/\eta;\Qell),
\end{equation}
where $\Bl(X/S)$ denotes Bloch intersection number and  $\Sw(X_{\eta}/\eta)$ the Swan conductor of $X_\eta$.
\end{conj}

\sssec{}

The intersection number $\Bl(X/S)$ is an algebro-geometric invariant of $X/S$: it was defined by Bloch in \cite[\S 1]{bl87}, with the notation $(\Delta_X, \Delta_X)_S$, as the top localized Chern number of the coherent sheaf $\Omega^1_{X/S}$.
In this paper, we will use a different characterization of $\Bl(X/S)$, due to Kato--Saito, which is more suitable for our computations, see Theorem \ref{thm:KS-interpretation of Bl}.

\sssec{} 

As mentioned above, the Swan conductor 
$$
\Sw(X_{\eta}/\eta;\Qell)
:=
\Sw \bigt{\uH^*_{\et}(X_{\bareta}, \Qell)}
$$ 
has an arithmetic origin: 
it vanishes if and only if the action of the inertia group on the $\ell$-adic cohomology of $X_{\bareta}$ is tame, so it is strictly related to \emph{wild ramification}. 
Thus, an intriguing aspect of \eqref{eqn:Bloch-formula} is that it relates topological, algebraic and arithmetic invariants of $X/S$.

\begin{rem}
It is convenient to rewrite \eqref{eqn:Bloch-formula} as
\begin{equation*} \label{eqn:true-BCF}
\Bl(X/S) 
= 
-
 \bigt{
 \Sw(X_{\eta}/\eta;\Qell) - \chi(X_s;\Qell) + \chi(X_{\bareta};\Qell)
 }.
\end{equation*}
The quantity
$$
\Sw(X_{\eta}/\eta;\Qell) - \chi(X_s;\Qell) + \chi(X_{\bareta};\Qell)
$$
is sometimes called the Artin conductor of $X/S$.
However, we prefer to set
$$
  \dimtot(X/S)
  :=
  \Sw(X_{\eta}/\eta) - \chi(X_s;\Qell) + \chi(X_{\bareta};\Qell)
$$
and call this number the \emph{total dimension of $X/S$}. With these conventions, Conjecture \ref{conj:Bloch} reads
$$
\Bl(X/S)
  =
  -\dimtot(X/S).
$$
\end{rem}

\sssec{} \label{sssec:state of art BCC}

Several cases of Conjecture \ref{conj:Bloch} have been established:
\begin{itemize}
\item
in his seminal paper \cite{bl87}, Bloch proves it for $X/S$ a family of curves.

\item relative dimension zero: in this case, BCC is the conductor discriminant formula from algebraic number theory.

\item when $S$ is of pure characteristic zero; this case can be extracted from work of Kapranov \cite{kapranov}. Also, this case follows from \cite{ks04} by combining their result with Hironaka resolution of singularities in characteristic zero.
    
\item 
in \cite{ks04}, Kato and Saito use logarithmic algebraic geometry to prove BCC under the hypothesis that $(X_s)_{\on{red}} \hto X$ is a normal-crossing divisor.

\item in \cite{sai21}, Saito develops the theory of characteristic cycles in positive characteristic, obtaining the proof of Conjecture \ref{conj:Bloch} in pure-characteristic, under the hypothesis that $X/S$ is projective.

\item In \cite{ab00}, Abbes highlights that a similar formula makes sense for all $S$-endomorphisms of $X$ and generalizes the proof of Bloch to give a formula valid for arithmetic surfaces with an $S$-automorphism. This point of view is adopted in \cite{ks04} too.

\end{itemize}
\sssec{}

In particular, the first item, combined with \cite{or03}, implies that the Deligne--Milnor conjecture is true in relative dimension $1$.
In fact, Orgogozo shows that the formula proposed by Bloch coincides with that proposed by Deligne in the situation of an isolated singularity.
However, like the Deligne--Milnor conjecture, Conjecture \ref{conj:Bloch} remains open in general.

\ssec{A generalization of Bloch conductor conjecture}

In this section, we state the second and third main results of this paper, which are a contribution to the list of Section \ref{sssec:state of art BCC}.
Our methods are well suited to tackle the more general situation where the map $X/S$ is not necessarily proper. This is the setting of the generalized Bloch conductor conjecture, stated below.

\sssec{}

It turns out that the number $\Bl(X/S)$ can be defined even when $p: X\to S$ is not proper: one just needs the singular locus of $p$ to be proper over $S$.
Accordingly, as suggested in \cite{or03}, Bloch conductor conjecture may be generalized as follows.

In what follows, we will use the notation $\Phi_{X/S}$ as a shortcut for $\Phi_p(\Ql{X})$, i.e. the sheaf of $\ell$-adic vanishing cycles with coefficients in $\Ql{X}$. 
Notice that $\Phi_{X/S}$ is supported on the singular locus $Z$ of $X/S$.
In particular, with our assumption:
$$
  \uH^*_{\et}(X_s,\Phi_{X/S})
  \simeq
  \uH^*_{\et,c}(X_s,\Phi_{X/S}),
$$
where the right hand side denotes cohomology with compact support.

\begin{conj}[Generalized Bloch conductor formula]
    Let $p:X\to S$ be flat morphism of finite type, generically smooth. Assume that $X$ is regular and that its singular locus is proper over $S$. 
    Then
\begin{equation} \label{eqn:gBCC in intro}
  \Bl(X/S)
      =
      -\dimtot \bigt{\uH^*_{\et}(X_s,\Phi_{X/S})}
      :=
      -\chi \bigt{\uH^*_{\et}(X_s,\Phi_{X/S})}- \Sw \bigt{\uH^*_{\et}(X_s,\Phi_{X/S})}.   
\end{equation}
\end{conj}

When $p:X\to S$ is proper
$$
  \Sw \bigt{\uH^*_{\et}(X_s,\Phi_{X/S})}
  =
  \Sw \bigt{\uH^*_{\et}(X_{\bareta},\Qell)}
$$
and 
$$
  \chi \bigt{\uH^*_{\et}(X_s,\Phi_{X/S})}
  =
  -\chi(X_s;\Qell)+\chi(X_{\bareta};\Qell),
$$
so that the generalized Bloch conductor formula agrees with the original one in this case.

\sssec{}

In the main body of the paper, we will prove this conjecture in the following two cases.

\begin{mainthm} \label{mainthm:hypersurface}
The generalized Bloch conductor conjecture holds true provided that the following two assumptions are satisfied:

\begin{itemize}

\item
 $X$ embeds as a hypersurface in a smooth $S$-scheme;
\item
the inertia group acts unipotently on $\Phi_{X/S}$.
\end{itemize}

\end{mainthm} 

\begin{mainthm} \label{mainthm:pure-char}
Assume that $S$ is of pure characteristic.
Then the generalized Bloch conductor conjecture holds true provided that the inertia group acts unipotently on $\Phi_{X/S}$.
\end{mainthm}

\begin{rem}
As in \cite{ab00}, we can consider a generalized Bloch conductor formula where endomorphisms other than the identity are allowed. Thus, we will actually prove a more general version of Theorem \ref{mainthm:hypersurface}, see Theorem \ref{main theorem 3}.
A similar modification works for Theorem \ref{mainthm:pure-char}, and it is left to the reader.
\end{rem}

\sssec{}

It turns out that Theorem \ref{mainthm:DM-unip} is an easy consequence of Theorem \ref{mainthm:hypersurface}.
This follows from \cite{or03} and from the fact that we can replace $U$ with any Zariski neighborhood of $x$: since $U$ is locally a hypersurface in a smooth $S$-scheme, Theorem \ref{mainthm:hypersurface} applies in this case.

\sssec{}

In the remainder of this introduction, we outline the proof of Theorem \ref{mainthm:hypersurface} (the proof of Theorem \ref{mainthm:pure-char} is largely parallel). There are two main steps, roughly corresponding to the two sides of \eqref{eqn:gBCC in intro}.

\ssec{Categorifying Bloch intersection number}

Our main construction shows that $\Bl(X/S)$ is induced by a dg-functor of differential-graded (dg) categories upon taking $K$-theory. This construction requires $X$ to be a presented as a hypersurface in a smooth $S$-scheme, but no unipotence assumption of the action of the inertia group.

\sssec{} 
Let $\uG_0(Y)$ denote the \emph{$\uG$-theory} of a scheme $Y$.
The starting point is the following crucial result.

\begin{thm}[{\cite[\S 5.1]{ks04}}]\label{thm:KS-interpretation of Bl}
For $n \gg 0$, the Tor-sheaves $\ul\Tor_{n}^{X \times_S X}(\Delta_X,E)$ are supported on the singular locus $Z$ of $X/S$. 
Moreover, the difference
$$
  (-1)^n 
  \Big[ 
  \ul\Tor_{n}^{X \times_S X}(\Delta_X,E)
  \Big]
  +
  (-1)^{n+1} 
  \Big[ 
  \ul\Tor_{n+1}^{X \times_S X}(\Delta_X,E)
  \Big]
$$
stabilizes for $n\gg 0$, providing a uniquely defined element of $G_0(Z)$.
\end{thm}

\sssec{}

In particular, we get a well defined map
\begin{equation}\label{defn:KS pairing}
 [[\Delta_X,-]]_S: G_0(X \times_S X)\rightarrow G_0(s)\simeq \bbZ   
\end{equation}
$$
  [E]
  \mapsto 
  (-1)^n 
  \deg
  \Big[ 
  \ul\Tor_{n}^{X \times_S X}(\Delta_X,E)
  \Big]
+
(-1)^{n+1} 
\deg
\Big[ 
\ul\Tor_{n+1}^{X \times_S X}(\Delta_X,E)
 \Big]
  \hspace{0.5cm} \text{(for } n\gg 0 \text{)}.
$$
The expression on the right is called the \emph{degree of the difference of the stable Tors} and it can be used to recover Bloch intersection number as follows.

\begin{thm}[{\cite[\S 5.1]{ks04}}]
The following identity holds
$$
  [[\Delta_X,\Delta_X]]_S=\Bl(X/S).
$$ 
\end{thm}

We will refer to the map $ [[\Delta_X,-]]_S$ as the \emph{Kato--Saito localized intersection product} and take $ [[\Delta_X,\Delta_X]]_S$ as the definition of Bloch intersection number.

\sssec{}

Now assume that $p:X \to S$ can be written as a hypersurface in a smooth $S$-scheme $P$, that is, $X$ is the zero locus of a section of a line bundle $L$ over $P$. 
We will use this extra piece of structure to categorify $ [[\Delta_X,-]]_S$.

\sssec{}

The appearance of the stable Tors in \eqref{defn:KS pairing} suggests that the dg-functor in question could be expressed via \emph{matrix factorizations} or, more generally, dg-categories of singularities. Recall that the dg-category of singularities $\Sing(W)$ of a bounded derived scheme $W$ is the quotient of $\Coh(W)$ by its full subcategory $\Perf(W)$, see Section \ref{sssec: def sing cat}.
In our case, consider $\Sing(X \times_S X)$: applying $\HK^{\bbQ}$ (homotopy-invariant non-connective rational K-theory), we obtain a natural map
$$
\uG_0(X \times_S X)
\to 
\HK_0^{\bbQ}(X \times_S X).
$$

\sssec{}

Hence, we look for a dg-functor out of $\Sing(X \times_S X)$. A naive guess is the pullback along the diagonal $\delta_X: X \hto X \times_S X$. This does not work, since $\delta_X$ is not quasi-smooth. To fix this, we use the presentation to define a derived thickening 
$$ 
d\delta_X: K(X,L,0) \hto X \times_S X
$$ 
of the diagonal, where $K(X,L,0)$ is the derived self-intersection of the zero section of $\restr L X$. Contrarily to $\delta_X$, the map $d\delta_X$ is quasi-smooth, hence the dg-functor $d\delta_X^*$ descends to a well-defined functor
\begin{equation} \label{eqn:intersect-diag}
\Sing(X \times_S X) 
\longto
\Sing \bigt{K(X,L,0)}
\end{equation}
between dg-categories of singularities.

\begin{rem}
By a version of Orlov's theorem (\cite{brtv18,pi22b}), the dg-category 
$\Sing \bigt{K(X,L,0)}$ is equivalent to the dg-category of $L$-twisted matrix factorizations on $X$.
\end{rem}

\sssec{} \label{sssec:degree for HK}

We will then check that \eqref{eqn:intersect-diag} lands in the full subcategory $\Sing \bigt{K(X,L,0)}_Z \subseteq \Sing \bigt{K(X,L,0)}$ consisting of those $L$-twisted matrix factorizations that are set-theoretically supported on the singular locus $Z$ of $X/S$. We will then observe that there is a canonical degree map
$$
\HK_0^{\bbQ}
\Bigt{ \Sing \bigt{K(X,L,0)}_Z }
\to \bbQ.
$$

\sssec{}

Denote by
\begin{equation} \label{eqn:integration}
\int_{X/S}:
\HK_0^{\bbQ}
\bigt{ 
\Sing(X  \times_S X) 
}
\xto{\HK_0^{\bbQ} (d\delta_X^*)} 
\HK_0^{\bbQ}
\Bigt{ \Sing \bigt{K(X,L,0)}_Z }
\to 
\bbQ
  \end{equation}
the resulting map. We will prove (Theorem \ref{thm: KS loc int prod = HK_0 int with the diagonal}) that this integration map coincides with the Kato--Saito localized intersection product.

\begin{thm}
For $E \in \Coh( X\times_S X)$, let $[E]$ be its class in $\HK_0^{\bbQ}\bigt{\Sing(X \times_S X)}$.
We have
    $$
      \int_{X/S}[E]=[[\Delta_X,E]]_S.
    $$
In particular, $\int_{X/S}[\Delta_X]$ equals Bloch intersection number $\Bl(X/S)$.
\end{thm}

\ssec{The categorical total dimension}

In the above step, we recovered Bloch intersection number by decategorifying our functor \eqref{eqn:intersect-diag}, and composing with a degree map.
We now decategorify \eqref{eqn:intersect-diag} in a different way and find (a number related to) the total dimension.

\sssec{}

This other decategorification procedure, constructed in \cite{brtv18}, goes under the name of \emph{$\ell$-adic realization of dg-categories}. Consider the symmetric monoidal dg-category $\shvl(S)$ of ind-constructible $\ell$-adic sheaves and let
$$
\Ql{S}(\beta):=\bigoplus_{j\in \bbZ}\Ql{S}(j)[2j],
$$
viewed as a commutative algebra object in $\shvl(S)$.
The $\ell$-adic realization constructed by \cite{brtv18} is a lax-monoidal $\infty$-functor
$$
\rl_S:\dgCat_S
\longto 
\Mod_{\Ql{S}(\beta)}
(\shvl(S))
$$
with the following properties:
\begin{itemize}
\item
it is compatible with filtered colimits;
\item
it is sends Drinfeld--Verdier localization sequences to fiber/cofiber sequences;
\item
for $u:Y\to S$ a quasi-compact quasi-separated $S$-scheme, we have
$$
\rl_S (\Perf(Y) )
\simeq 
u_* \Ql{Y} \otimes_{\Ql{S}} \Ql{S}(\beta).
$$
\end{itemize}

\sssec{}

The last item above implies that $\rl_S\bigt{\Perf(S)} \simeq \Ql{S}(\beta)$.
Consider now the dg-functor
$$
\Perf(S) 
\to
\Sing(X \times_S X)
$$
induced by pull-push along $S \xleftarrow{p} X \xto{\delta} X \times_S X$. Composing this arrow with \eqref{eqn:intersect-diag}, we obtain a dg-functor
$$
\Perf(S) 
\to
\Sing(X \times_S X)
\xto{{\eqref{eqn:intersect-diag}}}
\Sing \bigt{K(X, L,0)}_Z.
$$

\sssec{}

Applying $\rl_S$ to this composition, we find a map
\begin{equation}
\Ql{S}(\beta) = 
\rl_S 
\bigt{
\Perf(S) 
}
\longto
\rl_S 
\Bigt{ \Sing \bigt{K(X,L,0)}_Z }.
\end{equation}
Similarly to Section \ref{sssec:degree for HK}, there is a further canonical map
$$
\rl_S 
\Bigt{ \Sing \bigt{K(X,L,0)}_Z }
\to 
i_*\Ql{s}(\beta).
$$
Combining these, we obtain a $\Ql{S}(\beta)$-linear arrow
$$
\Ql{S}(\beta) = 
\rl_S 
\bigt{
\Perf(S) 
}
\to
\rl_S 
\Bigt{ \Sing \bigt{K(X,L,0)}_Z }
\to 
i_*\Ql{s}(\beta),
$$
which is determined in $\uH^0$ by a single element of $\Qell$.
We define the $\ell$-adic rational number $\dimtot^{\cat}(X/S)$ to be the negative of such element.

\sssec{}

From this, we immediately obtain the equality
$$
\int_{X/S} [\Delta_X ] = - \dimtot^{\cat}(X/S).
$$
Indeed, the two decategorifications $\HK_0^{\bbQ}$ and $\rl_S$ are related by {\TV}'s non-commutative $\ell$-adic Chern character, see \cite[\S 2.3]{tv22}; in the case at hand, this Chern character is simply the inclusion $\bbQ \subseteq \Qell$. 
In passing, notice that the left-hand side is an integer that is independent of $\ell$ and of the presentation of $X$ as a hypersurface, hence the same is true for the right-hand side.

In view of the above discussion, Theorem \ref{mainthm:hypersurface} is equivalent to the following.
\begin{conj}\label{conj: categorical Artin vs Artin}
Suppose that $X$ is presented as a hypersurface in a smooth $S$-scheme and that the inertia group acts unipotently on $\uH^*(X_{\bareta},\Qell)$. Then the categorical total dimension equals the classical one:
$$
\dimtot^{\cat}(X/S)= \dimtot(X/S).
$$
\end{conj}

\sssec{}

The last step of our proof amounts to proving this conjecture in the case of unipotent monodromy (we had not used this assumption yet).
Of course, one of the simplifications of the unipotent case is that the Swan conductor vanishes.\footnote{The proof in the non-unipotent case requires new ideas and will appear in \cite{bp24}.}

\ssec{Outline of the paper}

This paper is organized as follows:
\begin{itemize}
    \item In Section \ref{sec:preliminaries}, we introduce the necessary background and recall the non-commutative trace formula proven by {\TV}. We then construct an explicit duality datum for $\T = \Sing(X_s)$, the dg-category of singularities of the special fiber.
   
    \item In Section \ref{sec:K-intersection-thry}, we construct a categorification of the Kato--Saito localized intersection product (in particular, we obtain a categorification of Bloch intersection number).

    \item In Section \ref{sec:unip-monodromy} we use the constructions of the previous sections to prove Theorem \ref{mainthm:hypersurface}, which includes Theorem \ref{mainthm:DM-unip} as a special case.

    \item 
    Finally, in Section \ref{sec:pure-char}, we prove Theorem \ref{mainthm:pure-char}.
\end{itemize}

\sec*{Acknowledgements}

 We are deeply grateful to Bertrand To\"en and Gabriele Vezzosi for their influence through their works and several generous explanations.
 We would also like to thank Benjamin Hennion, Valerio Melani, Mauro Porta and Marco Robalo for interesting conversations. MP thanks Denis Nardin for a useful discussion on the difference between Koszul and symmetric algebras in the homotopical context. 
 
 Part of this work was done while MP was supported by the collaborative research center SFB 1085 \emph{Higher Invariants - Interactions between Arithmetic Geometry and Global Analysis} funded by the Deutsche Forschungsgemeinschaft.

 This project has received funding from the PEPS JCJC 2024.

\sec{Preliminaries}\label{sec:preliminaries}

The purpose of this section is to recollect the results of \cite{brtv18} and \cite{tv22} in order to fix some background and some notation.
In addition to this, we construct an explicit duality datum for a dg-category that will play a key role in the proofs: the dg-category $\sT := \Sing(X_s)$ of singularities of the special fiber.

\ssec{Trace formalism in non-commutative algebraic geometry}

\sssec{} 

We will use the theory of $\oo$-categories as developed in \cite{lu09,lu17}.
We always denote by $S=\Spec(A)$ a strictly henselian trait. No assumption is made on $S$: it can be of mixed or of pure characteristics.
     
\sssec{}

We denote by $\dgCat_A$ the $\oo$-category of small $A$-linear dg-categories up to Morita equivalence and by $\dgCAT_A$ the $\oo$-category of presentable $A$-linear dg-categories and continuous ($A$-linear) functors (see \cite{to07}). Both $\dgCat_A$ and $\dgCAT_A$ are symmetric monoidal under the tensor product $\otimes_A$ and the ind-completion functor
$$
  \wh{(-)} : \dgCat_A \longto \dgCAT_A
$$
is symmetric monoidal. So we regard $\dgCat_A$ as a non-full subcategory of $\dgCAT_A$ and call \emph{small} those morphisms in $\dgCAT_A$ that belong to $\dgCat_A$. 

\sssec{}

For $\sB$ a small monoidal $A$-linear dg-category, we denote by $\dgCat_{\B}$ the $\oo$-category of left $\B$-modules and $\sB$-linear functors.

A monoidal $A$-linear dg-category $\B$ determines a second monoidal $A$-linear dg-category $\B^{\rev}$, which has the same underlying dg-category as $\B$ but where the monoidal structure has been reversed: if $\star$ denotes the product in $\sB$, then 
$$
  b\star^{\rev} b'=b'\star b.
$$
Clearly, the $\oo$-category of right $\B$-modules $\dgCat^{\B}$ is equivalent to the $\oo$-category of left $\B^{\rev}$-modules.

\sssec{}

For $\B$ as above, the dg-category $\B^{\rev}\otimes_A \B$ is a monoidal $A$-linear dg-category (with the component-wise tensor structure) and it acts on $\sB$ from the left and from the right. 
We denote by $\sB^L$ (respectively, $\sB^R$)
the dg-category $\sB$ regarded as a left (respectively, right) $\B^{\rev}\otimes_A \B$-module.
Now, given a left $\B$-module $\T$ and a right $\B$-module $\T'$, the monoidal dg-category $\B^{\rev}\otimes_A \B$ naturally acts on $\T'\otimes_A\T$ and we have
$$
\T'\otimes_{\B}\T \simeq (\T\otimes_A\T') \otimes_{\B^{\rev}\otimes_A \B} \B^L.
$$

\sssec{}

It is well-known that for $\T \in \dgCat_A$, the ind-completion $\wh \T$ is dualizable with dual equal to $\wh{\T^{\op}}$. This implies that if $\T$ is a left $\B$-module, $\wh{\T^{\op}}$ is a right $\wh \B$-module with the dual action. 

\sssec{}

Denote by 
$$
\mu:\wh \B \otimes_{\wh A} \wh \T \to \wh \T
$$
$$
\mu^{\op}:
\wh{\T^{\op}} \otimes_{\wh A} \wh \B 
\to
\wh{\T^{\op}}
$$
the left action of $\wh \B$ on $\wh \T$ and the right action of $\wh B$ on $\wh{\T^{\op}}$.
We say that $\T$ is \emph{cotensored} over $\B$ if $\mu^{\op}$ is a small morphism, that is, if the right $\wh \B$-module structure on $\wh{\T^{\op}}$ comes from a right $\B$-module structure on $\T^{\op}$.

\sssec{}

Let $\mu^*: \wh \T \to \wh \B \otimes_{\wh A} \wh \T$ denote the right adjoint to $\mu$. It determines (by adjunction) a morphism $h: \wh \T^{\op} \otimes_{\wh A} \wh \T \to \wh \B$.
We say that $\T$ is \emph{proper} over $\B$ if $h$ is a small morphism (that is, $\T$ is enriched over $\B$).

\sssec{}

By \cite[Proposition 2.4.6]{tv22} if $\T \in \dgCat_{\B}$ is cotensored over $\B$, then $\wh \T$ has a right dual as a left $\wh \B$-module whose underlying dg-category is $\wh{\T^{\op}}$.

\begin{defn}[{\cite[Definition 2.4.7]{tv22}}]
We say that $\T$ is \emph{saturated over $\B$} if it is cotensored over $\B$ and if the two big morphisms above are small. In this case, we will denote them by $\ev_{\T/\B}$ and $\coev_{\T/\B}$, respectively.
We will write $\ev$ and $\coev$ if this will not cause confusion.
\end{defn}

\sssec{}

Define the \emph{Hochschild homology} of $\sB$ over $A$ to be the dg-category
$$
\HH_*(\sB/A)
:=
\B^R  \otimes_{\B^{\rev}\otimes_A \B}\B^L.
$$
Denote by 
$$
\ev^{\HH}
=
\ev^{\HH}_{\T/\B}:
\T^{\op}\otimes_{\B} \T 
\to
\HH_*(\B/A)
$$
the functor defined by the composition
$$
\T^{\op}\otimes_{\B} \T 
\simeq 
\bigt{\T\otimes_A \T^{\op}}
\usotimes_{\B^{\rev}\otimes_A \B} \sB^L 
\xto{\ev \otimes \id} 
\B^R  \otimes_{\B^{\rev}\otimes_A \B}\B^L
= 
\HH_*(\sB/A).
$$
Using this new functor, we can define non-commutative traces as follows.

\begin{defn}[{\cite[Definition 2.4.4]{tv22} and \cite[\S 4.2.1]{lu17}}] \label{defn: non-commutative trace}
 
Assume that $\T$ is saturated over $\B$ and let $f:T\to T$ be a $\B$-linear endomorphism. We define the \emph{non-commutative trace} of $f$ to be the object
$$
\Tr_{\B}(f:T)
\in 
\HH_*(\sB/A)
$$
corresponding to the composition
$$
  \Perf(A)
  \xto{\coev} 
  \T^{\op}\otimes_{\B} \T 
  \xto{\id \otimes f} 
  \T^{\op}\otimes_{\B} \T 
  \xto{\ev^{\HH}} 
  \HH_*(\B/A).
$$

\end{defn}

\ssec{K\"unneth formula for dg-categories of singularities}

Here we review an equivalence, due to \TV, see \cite[Theorem 4.2.1]{tv22}. It will play a crucial role in our computations.

\sssec{}\label{sssec: def sing cat}

For a (bounded, noetherian derived) scheme $W$ over $S$, we will consider:
              \begin{itemize}
                  \item $\Coh(W)$, the dg-category of complexes of $\CO_W$-modules with bounded and coherent total cohomology;
                  \item $\Perf(W)$, the dg-category of perfect complexes on $W$;
                  \item $\QCoh(W)$, the dg-category of quasi-coherent complexes on $W$;
              \end{itemize}
Moreover, for such a $W$, we will consider
$$
         \Sing(W):=\Coh(W)/\Perf(W),
$$
the dg-category of singularities of $W$. This dg-category vanishes if and only if $W$ is regular.

\sssec{}

Let $p:X \to S$ be as in the generalized Bloch conductor formula and consider the special fiber $p_s: X_s \to s$. 
The dg-category $\T:=\Sing(X_s) \in \dgCat_A$ is dualizable, and  our goal is to make the duality datum explicit.

\sssec{}

Consider the derived groupoid $G:=s \times_S s$, which acts naturally on $X_s$. We denote the action map by $\mu: G \times_s X_s \to X_s$.

We also set $\B^+ := \Coh(G)$: this is an algebra object of $\dgCat_A$ under convolution and it acts naturally on $\Coh(X_s)$.

\sssec{}

Set also $\B := \Sing(G) := \Coh(G)/\Perf(G)$ and $\T:= \Coh(X_s)/\Perf(X_s)$. The above $\B^+$-action on $\Coh(X_s)$ descends to an action of $\B$ on $\T$ (see \cite[Proposition 4.1.5 and \S 4.1.3]{tv22}).

By \cite[Proposition 4.1.7]{tv22}, we know that $\T$ is cotensored over $\B$. This means that the dual action morphism is a small morphism: for any $\phi \in \T^{\op}$ and $b \in \B$, the functor $\Hom_{\wh \T}(\phi, b \cdot -)$ preserves colimits.

\sssec{}

Since $\T$ is cotensored, we can form the tensor product 
$$
\T^{\op} \usotimes_{\B} \T 
\in \dgCat_A.
$$
There is an equivalence $\T^{\op} \usotimes_{\B} \T \to \Sing(X \times_S X)$ of $(A,A)$-bimodule dg-categories.
This equivalence comes from an equivalence
$$
\Coh(X_s)^{\op} \usotimes_{\B^+} \Coh(X_s)
\xto{\simeq}
\Coh(X \times_S X)_{X_s \times_s X_s},
$$
which in turn is induced by the functor
$$
\wt\ffF:
\Coh(X_s)^{\op} \otimes_{A} \Coh(X_s)
\longto
\Coh(X \times_S X)_{X_s \times_s X_s}
$$
$$
(E, F) 
\squigto
j_*( \bbD E \boxtimes_s F),
$$
where $j: {X_s \times_s X_s} \hto X \times_S X$ is the obvious closed embedding, $\bbD E := \ul\Hom(E, \O_{X_s})$ and $-\boxtimes_s -$ denotes the external tensor product relative to $s$.

Let us recall the following result proven by {\TV}.

\begin{thm}[{\cite[Theorem 4.2.1]{tv22}}] \label{thm:TV}
The above functor induces an equivalence
$$
\ffF:
\T^{\op} \usotimes_{\B} \T
\xto{ \;\;\simeq \;\;}
\Sing(X \times_S X).
$$
\end{thm}

\begin{cor}[{\cite[Proposition 4.3.1]{tv22}}]
  Let $\T$ and $\B$ be as above. Then $\T$ is saturated over $\B$.
\end{cor}

\begin{rem}
We believe that, though the theorem above is enough to conclude that $\T$ is saturated over $\B$ (i.e. that there exists a duality datum) and this is all is needed for the proofs in \cite{tv22}, it takes a bit of work to construct an \emph{explicit} duality datum by means of it.
\end{rem}

\ssec{An explicit duality datum for \texorpdfstring{$\T/\B$}{T/B}}\label{ssec: explicit duality datum}

We use the above equivalence to exhibit an explicit duality datum for $\T$ as a left $\B$-module. 

\sssec{}

Unsurprisingly, the candidate dual is the right $\B$-module $\T^\op$.
Then the evaluation must be a functor
$$
\ev:
\T \otimes_A \T^{\op} \to \B
$$
of $(\B,\B)$-bimodules, while the coevaluation must be a functor
$$
\coev:
\Perf(S) \to \T^\op \otimes_\B \T
$$
in $\dgCat_A$.
After exhibiting these functors, we will need to show that the compositions
\begin{equation} \label{eqn:zorro-1}
\T \simeq \T \otimes_A \Perf(A) 
\xto{\id \otimes \coev}
\T \otimes_A  \T^\op \otimes_\B \T
\xto{\ev \otimes \id}
\B \otimes_\B \T
\simeq
\T
\end{equation}
\begin{equation} \label{eqn:zorro-2}
\T^\op 
\simeq
\Perf(A) \otimes_A \T^{\op}
\xto{\coev \otimes \id}
 \T^\op \otimes_\B \T \otimes_A \T^{\op}
\xto{\id \otimes \ev}
\T^\op \otimes_\B  \B
\simeq
\T^{\op}
\end{equation}
are homotopic to the identity functors.

\sssec{}

To define the coevaluation, we use the equivalence of Theorem \ref{thm:TV}. Thus, $\coev$ is the functor
$$
\ell: 
\Perf(S) \to \Coh(X \times_S X ) \xto{proj} \Sing(X \times_S X)
$$
$$
M \squigto \delta_*(p^*(M)) \squigto proj \Bigt{\delta_*\bigt{p^*(M)}},
$$
where $\delta: X \to X \times_S X$ is the diagonal. We set $\Delta_X := \delta_*(\O_X) \in \Coh(X \times_S X)$, alerting the reader that we will often abuse notation and regard $\Delta_X$ as an object of $\Sing(X \times_S X)$ via the projection functor.

\sssec{} \label{sssec:notations}

Let us now construct the evaluation.
Consider the tautological maps
$$
X_s \times_S X_s
\xleftarrow{\;\; q \;\;}
X_s \times_X X_s
\xto{\;\; r \;\;}
G := s \times_S s.
$$ 
We denote by $q_1,q_2$ the compositions $X_s \times_X X_s \to X_s \times_S X_s \rr X_s$ of $q$ with the two projections.
Observe that we have an isomorphism
$$ 
G \times_s X_s
\simeq
X_s \times_X X_s
$$ 
$$
(g,x) \mapsto (g \cdot x, x).
$$
Under this isomorphism, the maps $q_1, q_2$ correspond to $\mu, pr: G \times_s X_s \rr X_s$ respectively, while $r$ corresponds to the projection $\pr_G: G \times_s X_s \to G$ onto $G$.

\sssec{} 

Now consider the functor
$$
\wt \ev:
\Coh(X_s) \otimes_A \Coh(X_s)^\op \to \B^+
$$
$$
(E, F) \squigto 
r_* q^* 
\bigt{ E \boxtimes_S \bbD F }.
$$
Here $-\boxtimes_S -$ denotes the external tensor product relative to $S$, i.e. $E\boxtimes_S E'= q_1^*E\otimes q_2^* E'$.
This functor does indeed land in $\B^+ = \Coh(G)$, since $X$ is regular and $r$ proper.

\begin{rem} \label{rem:ev}
In view of the above observations, an alternative way to write $\wt \ev$ is as
$$
(E, F)
\squigto
(\pr_G)_* (\mu^* E \otimes \pr^* (\bbD F)),
$$
where $\pr_G: G \times_s X_s \to G$ is the projection.
\end{rem}

\begin{lem} \label{lem:ev}
The above functor $\wt \ev$ descends to a functor
$$
\Sing(X_s) \otimes_A \Sing(X_s)^\op \to \B
$$
that we call $\ev$.
\end{lem}

\begin{proof}
We need to show that $\wt\ev(E,F) \in \Perf(G)$, as soon as at least one between $E$ and $F$ is perfect. 
Suppose that $F$ is perfect (the other case is completely analogous). Since $i: X_s \to X$ is affine, the functor $i_*: \QCoh(X_s) \to \QCoh(X)$ is conservative and thus $\Perf(X_s)$ is Karoubi-generated by the essential image of $i^*: \Perf(X) \to \Perf(X_s)$. 
In particular, we may assume that $F = i^* P$ for some $P \in \Perf(X)$. 
Thus, we need to prove that the object 
$$
M:=r_* q^* \bigt{ E \boxtimes_S \bbD F }
\simeq
r_* q^* 
\bigt{ E \boxtimes_S i^*(P^\vee)}
=
r_* q^* 
\bigt{ \pr_1^* E \otimes \pr_2^* i^*(P^\vee)}
$$
is perfect. Denoting by $q_1,q_2: X_s \times_X X_s \rr X_s$ the two projections, we obtain that
$$
M 
\simeq
r_*
\bigt{ q_1^* E \otimes q_2^* i^*(P^\vee)}
\simeq
r_*
\bigt{ q_1^* (E \otimes i^*(P^\vee)) },
$$
where the last step used $q_2 \circ i = q_1 \circ i$.
Now, $E':= E \otimes i^*(P^\vee)$ belongs to $\Coh(X_s)$, so it suffices to prove that $r_* \circ q_1^* (E')$ is perfect for any $E' \in \Coh(X_s)$. 
To this end, consider the \virg{swap} autoequivalence $\sigma:X_s\times_X X_s \simeq X_s\times_X X_s$. Let $E''= \sigma^* E'$. Then $q_1^*E'\simeq q_2^* E''$ and the isomorphism $X_s \times_X X_s \simeq X_s \times_s G$, together with base-change, implies that
$$
r_* \circ q_1^* (E')
\simeq
\CO_G \otimes_k H^*(X_s, E''),
$$
with $H^*(X_s, E'') = (p_{X_s})_*(E'')$. 
Since we may assume without loss of generality that $E''$ is supported on the singular locus of $X\to S$, which is proper over $s$, $H^*(X_s, E'')$ is a finite dimensional $k$-vector space and the assertion follows.
\end{proof}

\sssec{}

We now define the functor
$$
\wt\phi: 
\Coh(X_s) 
\otimes_A 
\Coh(X \times_S X)_{X_s \times_s X_s}
\longto
\QCoh( X_s)
$$
$$
(E, H) 
\squigto
(\pr_1)_*
\bigt{
\pr_2^* E \otimes j^* H
}.
$$
Our main computation is the following:

\begin{prop}\label{prop: comm diag for ev}
The diagram 
\begin{equation}  \label{diag:main}
\begin{tikzpicture}[scale=1.5]
\node (00) at (0,0) {$ \Coh(X_s) \otimes_A \Coh(X \times_S X)_{X_s \times_s X_s}$};
\node (10) at (4,0) {$\QCoh(X_s)$};
\node (01) at (0,1) {$ \Coh(X_s) \otimes_A \Coh(X_s)^\op \otimes_{\B^+} \Coh(X_s)$};
\node (11) at (4,1) {$\B^+ \otimes_{\B^+} \Coh(X_s)$};
\path[->,font=\scriptsize,>=angle 90]
(00.east) edge node[above] {$\wt \phi$}  (10.west); 
\path[->,font=\scriptsize,>=angle 90]
(01.east) edge node[above] {$\wt\ev \otimes \id $} (11.west); 
%%%VERT
\path[->,font=\scriptsize,>=angle 90]
(01.south) edge node[right] {$ \id \otimes \wt \ffF$} (00.north);
\path[->,font=\scriptsize,>=angle 90]
(11.south) edge node[right] {$\star$} (10.north);
\end{tikzpicture}
\end{equation}
is naturally commutative. 
Here $\star: \B^+\otimes_{\B^+}\Coh(X_s)\to \QCoh(X_s)$ denotes the dg-functor induced by the action of $\B^+$ on $\Coh(X_s)$.
\end{prop}

\begin{proof}

Let $E, F_1, F_2 \in \Coh(X_s)$. The upper path sends $(E,F_1, F_2)$ to 
$$
\wt \ev(E, F_1) \star F_2,
$$
which simplifies as
$$
M:=
(q_1)_*
\left(
r^* 
r_* q^* 
\bigt{ E \boxtimes_S \bbD F_1 }
\otimes 
q_2^*(F_2)
\right),
$$
where we recall that $q_1, q_2: X_s \times_{X} X_s \rr X_s$ are the two projections.

The bottom path sends $(E, F_1, F_2)$ to 
$$
N:=
(\pr_2)_* 
\bigt{
\pr_1^*(E) \otimes j^* j_* (\bbD F_1 \boxtimes_s F_2)
}.
$$
Our goal is construct a functorial isomorphism $M \simeq N$. 

We start by manipulating $M$. Using Section \ref{sssec:notations}, we have:
$$
M \simeq
\mu_*
\left(
r^* 
r_* q^* 
\bigt{ E \boxtimes_S \bbD F_1 }
\otimes 
\pr^* F_2
\right).
$$
Next, base-change along the fiber square
\begin{equation} 
\nonumber
\begin{tikzpicture}[scale=1.5]
\node (00) at (0,0) {$ G \times_s X_s $};
\node (10) at (3,0) {$G $ };
\node (01) at (0,1) {$ G \times_s X_s \times_s X_s $};
\node (11) at (3,1) {$G \times_s X_s $};
\path[->,font=\scriptsize,>=angle 90]
(00.east) edge node[above] {$r$}  (10.west); 
\path[->,font=\scriptsize,>=angle 90]
(01.east) edge node[above] {$\id_G \times \pr_1 $} (11.west); 
%%%VERT
\path[->,font=\scriptsize,>=angle 90]
(01.south) edge node[right] {$ \id_G \times \pr_2$} (00.north);
\path[->,font=\scriptsize,>=angle 90]
(11.south) edge node[right] {$r$} (10.north);
\end{tikzpicture}
\end{equation}
and the projection formula yield
\begin{eqnarray}
\nonumber
M 
& \simeq &
\mu_*
\Bigt{
(\id_G \times \pr_2)_* 
(\id_G \times \pr_1)^* q^* 
\bigt{ E \boxtimes_S \bbD F_1 }
\otimes 
\pr^* F_2
}
\\
\nonumber
& \simeq &
\mu_* (\id_G \times \pr_2)_* 
\Bigt{
(\id_G \times \pr_1)^* q^* 
\bigt{ E \boxtimes_S \bbD F_1 }
\otimes 
(\id_G \times \pr_2)^* 
\pr^* F_2
}.
\end{eqnarray}
We now use the observation of Section \ref{sssec:notations} to replace $q^* 
\bigt{ E \boxtimes_S \bbD F_1 }$ with $\mu^*E \otimes \pr^*(\bbD F_1)$. This yields
$$
M 
\simeq
\mu_* (\id_G \times \pr_2)_* 
\Bigt{
(\id_G \times \pr_1)^* \mu^* E
\otimes 
 (\id_G \times \pr_1)^* 
\pr^* (\bbD F_1 )
\otimes 
(\id_G \times \pr_2)^* 
\pr^* F_2
}.
$$
Now, it is obvious that 
$$
 (\id_G \times \pr_1)^* 
\pr^* (\bbD F_1 )
\otimes 
(\id_G \times \pr_2)^* 
\pr^* F_2
\simeq
\O_G \boxtimes \bbD F_1 \boxtimes F_2,
$$
where the external product is the one given by the three projections of $G\times_s X_s \times_s X_s$.

It remains to simplify the compositions $\mu \circ (\id_G \times \pr_i)$ for $i=1,2$.  To this end, we consider the diagonal action of $G$ on $X_s \times_s X_s$.
Denoting by $\nu$ the action map, it is clear that
$$
\mu \circ (\id_G \times \pr_i)
\simeq
\pr_i \circ \nu.
$$
for $i = 1,2$. All in all, we obtain
$$
M 
\simeq
(\pr_2)_* \nu_*
\Bigt{
\nu^* (\pr_1)^* E
\otimes 
\O_G \boxtimes \bbD F_1 \boxtimes F_2
}
\simeq
(\pr_2)_*
\Bigt{
(\pr_1)^* E
\otimes 
\nu_*(
\O_G \boxtimes \bbD F_1 \boxtimes F_2)
}.
$$
To conclude our proof, we just need to show that $\nu_*(\O_G \boxtimes -) \simeq j^* j_*(-)$. For this, we look at the fiber square
\begin{equation} 
\nonumber
\begin{tikzpicture}[scale=1.5]
\node (00) at (0,0) {$ X_s \times_s X_s $};
\node (10) at (3.5,0) {$X \times_S X$ };
\node (01) at (0,1) {$ G \times_s X_s \times_s X_s $};
\node (11) at (3.5,1) {$X_s \times_s X_s $};
\path[->,font=\scriptsize,>=angle 90]
(00.east) edge node[above] {$j$}  (10.west); 
\path[->,font=\scriptsize,>=angle 90]
(01.east) edge node[above] {$\pr_G \times \id_{X_s \times_s X_s} $} (11.west); 
%%%VERT
\path[->,font=\scriptsize,>=angle 90]
(01.south) edge node[right] {$ \nu$} (00.north);
\path[->,font=\scriptsize,>=angle 90]
(11.south) edge node[right] {$j$} (10.north);
\end{tikzpicture}
\end{equation}
and apply base-change.
Since all the equivalences in the steps above are functorial (base-change equivalences and projection formulas), we are done.
\end{proof}

\begin{cor}
The essential image of $\wt\phi$ is contained in $\Coh(X_s)$. Thus, from now on we consider $\wt\phi$ as a functor 
$\wt\phi: \Coh(X_s) 
\otimes_A 
\Coh(X \times_S X)_{X_s \times_s X_s}
\to
\Coh( X_s)$.
\end{cor}

\begin{proof}
Recall that the $\B^+$-action functor $\B^+ \otimes_A \Coh(X_s) \to \QCoh(X_s)$ lands in $\Coh(X_s)$. Then the assertion follows from the commutativity of \eqref{diag:main} and the fact that $\wt \ffF$ Karoubi-generates the target.
\end{proof}

\begin{cor}
The functor $\wt\phi$ descends to a functor 
$$
\phi:
\Sing(X_s) \otimes_A \Sing (X \times_S X)
\longto
\Sing(X_s).
$$
\end{cor}

\begin{proof}
Using the commutativity of \eqref{diag:main} and the equivalence $\wt \ffF$ again, it suffices to prove the following claim. Given three objects $F_0,F_1, F_2$ in $\Coh(X_s)$, the object $\wt\ev(F_0, F_1) \star F_2$ is perfect as soon as one among the $F_i$'s is.
Since the $\B^+$-action on $\Coh(X_s)$ preserves $\Perf(X_s)$, the assertion is clear in the case $F_2$ is perfect. In the other two cases, Lemma \ref{lem:ev} guarantees that $\wt\ev(F_0,F_1) \in \Perf(G)$ and we are done.
\end{proof}

\begin{cor}
The diagram 
\begin{equation}  \label{diag:main for sing}
\begin{tikzpicture}[scale=1.5]
\node (00) at (0,0) {$ \Sing(X_s) \otimes_A \Sing(X \times_S X)$};
\node (10) at (4,0) {$\Sing(X_s)$.};
\node (01) at (0,1) {$ \Sing(X_s) \otimes_A \Sing(X_s)^\op \otimes_{\B} \Sing(X_s)$};
\node (11) at (4,1) {$\B \otimes_{\B} \Sing(X_s)$};
\path[->,font=\scriptsize,>=angle 90]
(00.east) edge node[above] {$\phi$}  (10.west); 
\path[->,font=\scriptsize,>=angle 90]
(01.east) edge node[above] {$\ev \otimes \id $} (11.west); 
%%%VERT
\path[->,font=\scriptsize,>=angle 90]
(01.south) edge node[right] {$ \id \otimes \ffF$} (00.north);
\path[->,font=\scriptsize,>=angle 90]
(11.south) edge node[right] {$\star$} (10.north);
\end{tikzpicture}
\end{equation}
commutes naturally.

\end{cor}

\begin{lem}\label{lem: wt phi comp diag = id}
The functor
$$
 \phi(-, \Delta_X) : \Sing(X_s) \to \Sing(X_s)
$$
is naturally homotopic to $\id_{\Sing(X_s)}$.
\end{lem}

\begin{proof}
We will prove a stronger statement.
By Proposition \ref{prop: comm diag for ev} we have the commutative diagram
\begin{equation}
\nonumber
\begin{tikzpicture}[scale=1.5]
\node (00) at (0,0) {$ \Coh(X_s) \otimes_A \Coh(X \times_S X)_{X_s \times_s X_s}$};
\node (10) at (4,0) {$\Coh(X_s)$};
\node (01) at (0,1) {$ \Coh(X_s) \otimes_A \Coh(X_s)^\op \otimes_{\B^+} \Coh(X_s)$};
\node (11) at (4,1) {$\B^+ \otimes_{\B^+} \Coh(X_s)$};
\node (-01) at (0,-1) {$\Coh(X_s) \otimes_A \Coh(X \times_S X)$};

\node (-11) at (4,-1) {$\QCoh(X_s).$};

\path[->,font=\scriptsize,>=angle 90]
(00.east) edge node[above] {$\wt \phi$}  (10.west); 
\path[->,font=\scriptsize,>=angle 90]
(01.east) edge node[above] {$\wt\ev \otimes \id $} (11.west); 

\path[->,font=\scriptsize,>=angle 90]
(-01.east) edge node[above] {$\wt \phi$} (-11.west);
%%%VERT
\path[->,font=\scriptsize,>=angle 90]
(01.south) edge node[right] {$ \id \otimes \wt \ffF$} (00.north);
\path[->,font=\scriptsize,>=angle 90]
(11.south) edge node[right] {$\star$} (10.north);

\path[->,font=\scriptsize,>=angle 90]
(00.south) edge node[right] {$id \otimes \mathit{incl}$} (-01.north);

\path[->,font=\scriptsize,>=angle 90]
(10.south) edge node[right] {$\mathit{incl}$} (-11.north);
\end{tikzpicture}
\end{equation}
Unraveling the definition, we see that
$$
\wt \phi (-, \Delta_X) 
\simeq
(\pr_2)_*
\bigt{
\pr_1^* (-)
\otimes
j^* (\Delta_X)
}.
$$
Now, observing that the diagram
\begin{equation} 
\nonumber
\begin{tikzpicture}[scale=1.5]
\node (00) at (0,0) {$ X $};
\node (10) at (2,0) {$X \times_S X$ };
\node (01) at (0,1) {$ X_s $};
\node (11) at (2,1) {$X_s \times_s X_s $};
\path[->,font=\scriptsize,>=angle 90]
(00.east) edge node[above] {$\delta$}  (10.west); 
\path[->,font=\scriptsize,>=angle 90]
(01.east) edge node[above] {$\delta_{X_s} $} (11.west); 
%%%VERT
\path[->,font=\scriptsize,>=angle 90]
(01.south) edge node[right] {$ i $} (00.north);
\path[->,font=\scriptsize,>=angle 90]
(11.south) edge node[right] {$j$} (10.north);
\end{tikzpicture}
\end{equation}
is a (derived) fiber square, we get that
$j^* \Delta_X \simeq \Delta_{X_s} := (\delta_s)_*(\O_{X_s})$ and the assertion follows from the projection formula.

To conclude, observe that $X/S$ is generically smooth and so the functor $\Coh(X\times_S X)_{X_s \times_s X_s}\hookrightarrow \Coh(X\times_S X)$ induces an equivalence
$$
  \Sing(X\times_S X)_{X_s \times_s X_s}\simeq \Sing(X\times_S X).
$$
Thus the claim for the singularity category follows.
\end{proof}

\sssec{}

We can now conclude the proof that the pair $(\ev,\coev)$ is a duality datum for $\T$ over $\B$.

\begin{prop}\label{prop: explicit duality datum for T/S}
The functors
$$
\coev: \Perf(S) \to \T^{\op}\otimes_{\B}\T, \hspace{0.5cm} \ev: \T \otimes_A \T^{\op}\to \B
$$
defined above form a duality datum for the left $\B$-module $\T$.
\end{prop}

\begin{proof}
It follows immediately from \eqref{diag:main for sing} and Lemma \ref{lem: wt phi comp diag = id} that the composition
$$
\T \simeq \T \otimes_A A
\xto{\id \otimes \coev} \T \otimes_A \T^{\op} \otimes_{\B} \T \xto{\ev \otimes \id} \B \otimes_{\B} \T \simeq \T
$$
is homotopic to the identity.
The proof that the composition
$$
\T^{\op} \simeq A \otimes_A \T^{\op}\xto{\coev \otimes \id} \T^{\op} \otimes_{\B} \T \otimes_A \T^{\op} \xto{\id \otimes \ev} \T^{\op} \otimes_{\B} \B \simeq \T^{\op}
$$
is homotopic to the identity is similar and left to the reader.
\end{proof}

\begin{rem}\label{rem: endomorphism induced by push and pull on T}
Let $f:X\to X$ be an $S$-linear endomorphism which preserves the singular locus. 
Observe that $f$ is quasi-smooth (because $X$ is regular). 
Therefore, it induces an endofunctor
$$
  (f_s)_* : \T \to \T,
$$
where $f_s$ denotes the endomorphism of $X_s$ induced by $f$. 
It is easy to check that $(f_s)_*$ is $\B$-linear and that the diagram
\begin{equation} 
\nonumber
\begin{tikzpicture}[scale=1.5]
\node (00) at (0,0) {$ \Sing(X\times_S X) $};
\node (10) at (2.5,0) {$ \Sing(X \times_S X)$ };
\node (01) at (0,1) {$ \T^\op \otimes_{\B} \T $};
\node (11) at (2.5,1) {$ \T^\op \otimes_{\B} \T $};
\path[->,font=\scriptsize,>=angle 90]
(00.east) edge node[above] {$(\id \times f)_*$}  (10.west); 
\path[->,font=\scriptsize,>=angle 90]
(01.east) edge node[above] {$\id \otimes (f_s)_*$} (11.west); 
%%%VERT
\path[->,font=\scriptsize,>=angle 90]
(01.south) edge node[right] {$ \ffF $} (00.north);
\path[->,font=\scriptsize,>=angle 90]
(11.south) edge node[right] {$ \ffF $} (10.north);
\end{tikzpicture}
\end{equation}
is commutative.
In particular, we obtain that the composition
$$
  \Perf(S) \xto{\coev} \T^\op \otimes_{\B} \T \xto{id\otimes (f_s)_*}\T^\op \otimes_{\B} \T
$$
corresponds to the dg functor
$$
  \Perf(S) \to \Sing(X\times_S X)
$$
determined by the structure sheaf of the \emph{graph} of $f$.
\end{rem}

\sec{K-theoretic intersection theory on arithmetic schemes}\label{sec:K-intersection-thry}

In this section, we begin the proof of Theorem \ref{mainthm:hypersurface}.
Using the presentation of $X$ as a hypersurface, we construct a dg-functor whose decategorification recovers Bloch intersection number $\Bl(X/S)$. For this construction, we do not need the unipotence hypothesis of Theorem \ref{mainthm:hypersurface}.

\ssec{ 
Construction of the dg-functor of \virg{intersection with the diagonal}
}
\label{ssec:construction intersection with diag}

\sssec{}

We need to fix some notation.
For $W$ an $S$-scheme and $L$ a line bundle on it, we denote by $K(W,L,0)$ the derived self-intersection of the zero section of $W$ in the total space of $L$. The letter $K$ stands for \virg{Koszul}. 
When the line bundle is trivial, we omit it from the notation: in other words, we define
$$
K(W,0) 
:= 
W \times_{\bbA^1_W} W.
$$
When $W$ is regular, the scheme $K(W,L,0)$ is quasi-smooth over $W$.

\sssec{} \label{sssec:Orlov}

We will consider the dg-category of singularities $\Sing \bigt{K(W,L,0)}$. 
By a version of Orlov's theorem, $\Sing \bigt{K(W,L,0)}$ is equivalent to the dg-category $\MF(W,L,0)$ of $L$-twisted matrix factorizations on $W$. This equivalence
$$
  \Xi:
\Sing \bigt{K(W,L,0)}
  \xto{\;\;\sim \;\;}
  \MF(W,L,0)
$$
is induced by the dg-functor
$$
  \Xi^+:
 \Coh \bigt{K(W,L,0)}
  \to
  \MF(W,L,0)
$$
$$
  (M,d,h)
  \mapsto
  \Biggl [
  \dots
  \xto{d+h}
  \underbrace{\bigoplus_{i\in \bbZ}M^{2i+j}\otimes L^{\otimes i}}_{\on{deg} j}
  \xto{d+h}
  \underbrace{\bigoplus_{i\in \bbZ}M^{2i+j+1}\otimes L^{\otimes i}}_{\on{deg} j+1}
  \xto{d+h}
  \dots
  \Biggr ].
$$
Here $(M,d)$ is a bounded complex of coherent sheaves on $W$ and $h: M \to M[-1] \otimes L$ is such that $h^2 = 0$ (super)commuting with $d$.

\sssec{} 

We now impose the first hypothesis of Theorem \ref{mainthm:hypersurface}: we write $X=\on{V}(s)$, for a global section $s\in \uH^0(P,L)$ of a line bundle $L$ on a smooth $S$-scheme $P$.
This presentation yields the following two fiber squares:
    \begin{equation*}
        \begin{tikzpicture}[scale=1.5]
            \node (02) at (0,2) {$X\times_PX$};
            \node (12) at (2,2) {$P$};
            \node (01) at (0,1) {$X\times_SX$};
            \node (11) at (2,1) {$P\times_SP$};
            \node (00) at (0,0) {$P \times_S P$};
            \node (10) at (2,0) {$L \times_S L.$};
            %% horiz maps %%
            \draw[->] (02) to node[above] {${}$} (12);
            \draw[->] (01) to node[above] {$$} (11);
            \draw[->] (00) to node[above] {$0 \times 0$} (10);
            %% vertical maps %%
            \draw[->] (02) to node[left] {$d\delta_X$} (01);
            \draw[->] (12) to node[right] {$\delta_P$} (11);
            \draw[->] (01) to node[above] {$$} (00);
            \draw[->] (11) to node[right] {$s \times s$} (10);
        \end{tikzpicture}
    \end{equation*}

Observe that the diagonal $\delta_X: X \to X \times_S X$ factors as the composition
\begin{equation} \label{eqn:factoriz of diagonal}
X 
\xto{\iota}
K(X,L,0)
\simeq 
X\times_PX
\xto{d \delta_X}
X \times_S X,   
\end{equation}
with $\iota$ a derived thickening of $X$. We say that $d\delta_X$ is a \emph{derived enhancement} of $\delta_X$.

\sssec{}

By construction, the closed embedding $d\delta_X$ is quasi-smooth (while $\delta_X$ is not). 
This is an immediate consequence of the fact that $d\delta_X$ is the pullback of the diagonal $\delta_P:P \to P\times_SP$, which is quasi-smooth since $P$ is assumed to be smooth over $S$.
Hence, the quasi-coherent pullback dg-functor
$d\delta_X^*$ preserves both coherent and perfect objects, thereby inducing a dg-functor
\begin{equation} \label{eqn:d-delta pullback for Sing}
d\delta_X^*
:
\Sing(X \times_S X)
\to 
\Sing \bigt{K(X,L,0)}
\end{equation}
between the singularity categories.

\begin{lem}\label{lem: pullback along derived pullback factors through MF with support on the singular locus}
The dg-functor \eqref{eqn:d-delta pullback for Sing} factors through the full subcategory 
$$
\Sing \bigt{K(X,L,0)}_Z 
\subseteq
\Sing \bigt{K(X,L,0)},
$$
where $Z \subseteq X$ denotes the singular locus of $X/S$.
\end{lem}

\begin{proof}
Recall that $\Sing(K(X,L,0))_Z$ is defined as the kernel of the restriction functor
\begin{equation*}
\Sing \bigt{K(X,L,0)}
\rightarrow \Sing \bigt{K(U,\restr L 
 U,0)},
\end{equation*}
with $j: U \hto X$ the open complement of $Z \subseteq X$. 
Thus, we have to show that the composition
\begin{equation*}
    \Sing(X\times_S X)
    \xto
    {d \delta_X^*} 
\Sing \bigt{K(X,L,0)}
\xto{j^*}
\Sing \bigt{K(U,\restr L U,0)}
\end{equation*}
is identically zero. We will show, equivalently, that the composition
\begin{equation*}
    \Coh (X\times_S X)
    \xto
    {d \delta_X^*} 
\Coh \bigt{K(X,L,0)} 
\xto{j^*}
\Coh \bigt{K(U,\restr L U,0)}
\end{equation*}
factors through 
$$
\Perf \bigt{K(U,\restr L U,0)}
\subseteq 
\Coh \bigt{K(U,\restr L U,0)} .
$$
The following square is clearly commutative:
\begin{equation*}
    \begindc{\commdiag}[15]
      \obj(-50,15)[1]{$\Coh(X\times_SX)$}
      \obj(40,15)[2]{$\Coh \bigt{K( X,L,0)}$}
      \obj(-50,-15)[3]{$\Coh(U\times_SU)$}
      \obj(40,-15)[4]{$\Coh \bigt{K( U, \restr L U,0)}.$}
      \mor{1}{2}{$d\delta_X^*$}
      \mor{1}{3}{$(j\times_S j)^*$}
      \mor{2}{4}{$j^*$}
      \mor{3}{4}{$ d\delta_U^*$}
    \enddc
\end{equation*}
To conclude, observe that $U\times_SU$ is a regular scheme (since $U \to S$ is smooth), hence 
$$
\Coh(U\times_S U) 
\simeq 
\Perf(U\times_SU)
$$
and the lemma follows from the fact that $d\delta_U^*$ preserves perfect complexes.
\end{proof}

\begin{defn}
We define 
\begin{equation*}
    (-,\Delta_X):\Sing(X\times_S X)
    \rightarrow 
    \Sing \bigt{K(X,L,0)}_Z 
\end{equation*}
to be the dg-functor obtained from the above lemma. We will refer to it as the \emph{intersection with the diagonal}.
\end{defn}

\ssec{Motivic realization of dg-categories}

\sssec{}

Let $\SH_{S}$ denote the stable homotopy category of schemes introduced by Morel and Voevodsky in \cite{mv99} (see \cite{ro15} for an $\oo$-categorical version of the construction). 
It is a stable and presentable symmetric monoidal $\oo$-category. 

\sssec{}

For $Y$ a quasi-compact quasi-separated $S$-scheme, denote by $\BU_Y\in \SH_Y$ the spectrum of non-connective homotopy-invariant algebraic K-theory.

\sssec{}

Recall that in \cite{brtv18} the \emph{motivic realization of dg-categories} was defined. This is a lax-monoidal functor
$$
  \Mv_S:\dgCat_A \rightarrow \Mod_{\BU_S}(\SH_{S})
$$
that satisfies the following properties.
\begin{enumerate}

\item 
If $q:Y\rightarrow S$ is a quasi-compact, quasi-separated $S$-scheme, then $\Mv_S\bigt{\Perf(Y)}\simeq q_*\BU_Y$.

\item 
In particular, $\Mv_S\bigt{\Perf(S)}\simeq \BU_S$.

\item 
$\Mv_S$ preserves filtered colimits.

\item 
$\Mv_S$ sends exact sequences of dg-categories to exact triangles.

\end{enumerate}

\sssec{}

By tensorization with $\HQ$ (the spectrum of rational singular cohomology), we obtain a similar $\oo$-functor
\begin{equation*}
    \Mv_{\bbQ, S}:\dgCat_A \rightarrow \Mod_{\BU_{S,\bbQ}}(\SH_{S}),
\end{equation*}
where $\BU_{S,\bbQ}=\BU_S\otimes \HQ$ is the spectrum of non-connective homotopy-invariant \emph{rational} K-theory.

\ssec{The integration map}

We wish to \virg{extract} Bloch intersection number from our dg-functor $(-,\Delta_X)$. To this end, we need to study its motivic realization. Specifically, we explicitly compute the motive of $\MF(X,L,0)_Z$, as well as that of $\MF(S,0)_s$. These computations will reveal the existence of a natural map
$$
\Mv_{\bbQ,S}\bigt{\MF(X,L,0)_Z}
\to 
\Mv_{\bbQ,S}\bigt{\MF(S,0)_s}.
$$
Combining this map with the motivic realization of $(-,\Delta_X)$, we obtain our \emph{integration map}, which recovers the localized intersection product of Kato--Saito (and, in particular, Bloch intersection number).

\begin{prop}\label{prop: MF(X,C,0)_Z decomposes as a direct sum in motives}

Let $i:Z\hookrightarrow X$ denote the singular locus of $X/S$.
There is an equivalence
\begin{equation} \label{eqn:motive of MF(X,L,0)Z}
    \Mv_{\bbQ,S}\bigt{\MF(X,L,0)_Z}
    \simeq 
    p_*(i_*i^!\BU_{\bbQ,X}\oplus i_*i^!\BU_{\bbQ,X}[1])
\end{equation}
of $\BU_{\bbQ,S}$-modules.
Moreover, 
$$
\HK_0^{\bbQ}\bigt{\MF(X,L,0)_Z}\simeq \HK_0^{\bbQ}\bigt{\Coh(Z)}\simeq \uG_0(Z)\otimes \bbQ.
$$
\end{prop}

\begin{proof}
Since $\Mv_{\bbQ,S}$ sends localization sequences to exact triangles, $\Mv_{\bbQ,S}\bigt{\MF(X,L,0)_Z}$ is equivalent to the fiber of
\begin{equation*}
    \Mv_{\bbQ,S}\bigt{\MF(X,L,0)}\rightarrow \Mv_{\bbQ,S}\MF \bigt{U,\restr L U,0)}.
\end{equation*}
In view of \cite[\S3]{pi22}, we know that
\begin{equation*}
    \Mv_{\bbQ,S}\bigt{\MF(X,L,0)}
    \simeq 
    p_*\coFib(\BU_{\bbQ,X}
    \xto
    {\id-L^\vee}
    \BU_{\bbQ,X}),
\end{equation*}
\begin{equation*}
    \Mv_{\bbQ,S}\bigt{\MF(U,L,0)}
    \simeq 
    (\restr p U)_*\coFib(\BU_{\bbQ,U}
    \xto
    {\id-\restr {L^\vee} U}
    \BU_{\bbQ,U}).
\end{equation*}
By \cite[Lemma 5.1.3]{ks04}, the restriction $\restr{L^\vee} Z$ is equivalent to $\CO_Z$ in $\uG$-theory, so we obtain the claim by combining the above equivalences.
The last formula of the proposition is an easy consequence of the first one.
\end{proof}

\sssec{}

Let $i_S:s\to S$ be the embedding of the special point of $S$.
The proof above shows that
\begin{equation} \label{eqn:motive of MF(S,0)s}
\Mv_{\bbQ,S}\bigt{\MF(S,0)_s}
\simeq
(i_{S})_*\circ i_S^!(\BU_{\bbQ,S})\oplus (i_S)_* \circ i_S^!(\BU_{\bbQ,S})[1]   
\end{equation}
and
$$
\HK_0^{\bbQ}\bigt{\MF(S,0)_s}
\simeq 
\HK_0^{\bbQ} \bigt{\Coh(s)}
\simeq 
\bbQ.
$$

\sssec{}

Thanks to the regularity of $X$, we have a canonical isomorphism $\BU_{\bbQ,X}\simeq p^!\BU_{\bbQ,S}$, see \cite[Lemma 3.3.2]{tv22}. 
Consider now the natural arrow
\begin{equation} \label{eqn:natural arrow motives}
    p_* i_*  i^!(\BU_{\bbQ,X})
    \simeq 
   (i_S)_*  p_{Z!}p_Z^!i_S^!(\BU_{\bbQ,S})
    \to 
    (i_S)_* i_S^!(\BU_{\bbQ,S}),
\end{equation}
obtained by adjunction, where $p_Z$ denotes the morphism $Z\to s$.
Notice that $i_S \circ p_Z = p\circ i$.
Combining \eqref{eqn:natural arrow motives} with the equivalences \eqref{eqn:motive of MF(X,L,0)Z} and \eqref{eqn:motive of MF(S,0)s}, we obtain a natural arrow
\begin{equation} \label{eqn:deg-for-motives}
\wt{\deg}: \Mv_{\bbQ,S}\bigt{\MF(X,L,0)_Z}
    \to 
     \Mv_{\bbQ,S}\bigt{\MF(S,0)_s}
\end{equation}
of $\BU_{\bbQ,S}$-modules. 
In $\uG$-theory, this maps corresponds to the degree map 
$$
\deg= (p_Z)_*: \uG_0(Z) \otimes \bbQ \to \uG_0(s) \otimes \bbQ,
$$ 
which is well-defined since $Z/s$ is proper.

\sssec{}

We define the \emph{integration map in $\BU_{\bbQ,S}$-modules}
\begin{equation*}
\int_{X/S}^{\Mv_{\bbQ,S}}:
\Mv_{\bbQ,S}\bigt{\Sing(X\times_S X)}
\rightarrow 
\Mv_{\bbQ,S}\bigt{\MF(S,0)_s}
\end{equation*}
to be the composition
$$
\Mv_{\bbQ,S}\bigt{\Sing(X\times_S X)}
\xto
{
\Mv_{\bbQ,S}(-,\Delta_X)
}
\Mv_{\bbQ,S}\bigt{\MF(X,L,0)_Z}
\xto{
\wt{\deg}
}
\Mv_{\bbQ,S}\bigt{\MF(S,0)_s}.
$$
We then define the integration map to be the map induced by $\int_{X/S}^{\Mv_{\bbQ,S}}$ in $\HK$-theory:
\begin{equation*}
\int_{X/S}:
\HK^{\bbQ}_0
\bigt{\Sing(X\times_S X)}
\xto
{
\HK^{\bbQ}_0 (-,\Delta_X)
}
\HK^{\bbQ}_0\bigt{\MF(X,L,0)_Z}
\simeq 
\uG_0^{\bbQ}(Z)
\xto{
{\deg}
}
\uG_0^{\bbQ}(s) = \bbQ.
\end{equation*}

\sssec{}

The canonical dg-functor $\Coh \bigt{X \times_S X} \to \Sing \bigt{X \times_S X}$ induces a canonical map 
$$
\uG_0^{\bbQ}(X \times_S X) \to \HK^{\bbQ}_0
\bigt{\Sing(X\times_S X)}. 
$$
Hence, we can apply $\int_{X/S}$ to classes in $\uG_0^{\bbQ}(X \times_S X)$. The following is the main result of this section.

\begin{thm}\label{thm: KS loc int prod = HK_0 int with the diagonal}
Our integration map agrees with the localized intersection product of Kato--Saito: for $[E]\in G_0^{\bbQ}(X \times_S X)$, we have
$$
\int_{X/S}[E]=[[\Delta_X,E]]_S.
$$
\end{thm}

The proof depends on the following more refined result, which shows that our construction agrees with that of Kato--Saito \emph{before} taking the degree.

\begin{prop}\label{prop: KS loc int prod = HK_0 int with the diagonal}
Let
\begin{equation*}
    [\Delta_X,-]: \uG_0(X\times_SX)\otimes \bbQ \to \HK^{\bbQ}_0\bigt{\Sing(X\times_S X)} \rightarrow \HK^{\bbQ}_0\bigt{\MF(X,L,0)_Z}
\end{equation*}   
denote the morphism induced by $\Mv_{\bbQ,S}(-,\Delta_X)$. Then
\begin{equation*}
    [\Delta_X,-]=[[X,-]]_{X\times_S X},  
\end{equation*}
where the right-hand side was defined by Kato--Saito in \cite[Definition 5.1.5]{ks04}.
\end{prop}

\begin{proof}
 It is proven in \cite{ks04} that, for $n\gg 0$, there are equivalences 
$$
  \ul{\Tor}_n^{X\times_SX}(\Delta_X,E)
  \simeq
  \ul{\Tor}_{n-2}^{X\times_SX}(\Delta_X,E)\otimes L
$$
of coherent $\ccO_X$-modules supported on $Z$.
Accordingly, Kato--Saito defined the element 
$$
[[X,E]]_{X\times_SX}
\in \uK_0(X)_Z/\langle 1-[\restr L Z] \rangle
\simeq 
\uG_0(Z)
$$
as the difference
$$
[[X,E]]_{X\times_SX}
   := 
    [\ul \Tor^{X\times_SX}_{2n}(\Delta_X,E)]-
  [\ul \Tor^{X\times_SX}_{2n+1}(\Delta_X,E)] 
   =
[\CH^{-2n}(\delta_X^*E)]-
[\CH^{-2n-1}(\delta_X^*E)].
$$
We will express these Tor-sheaves in terms of the derived enhancement $d\delta_X$ of $\delta_X$.

\medskip

The factorization $\delta_X\simeq d\delta_X \circ \iota$, see \eqref{eqn:factoriz of diagonal}, yields
$$
\delta_{X*}\delta_X^*E 
\simeq
d\delta_{X*}(
d\delta_X^*
(E)
\otimes 
\iota _*\CO_X).
$$
Let $\psi: K(X,L,0) = X \times_P X \to P$ denote the obvious map.
The object $\iota _*\CO_X \in \Coh \bigt{K(X,L,0)}$ can be resolved as
$$
  \iota_*\CO_X
  \simeq
  \Bigl [
  \dots 
  \xto{\epsilon} 
  \psi^* (L^{\vee}[1])^{\otimes 2}
  \xto{\epsilon} 
  \psi^* L^{\vee}[1]
  \xto{\epsilon} 
 \CO_{ K(X,L ,0)}
  \Bigr ],
$$
so that $d\delta_X^*E
\otimes  \iota_*\CO_X$ is equivalent to the totalization of the cochain complex
$$
  \dots
  \xto{h}
  d\delta_X^*E\otimes (L^{\vee}[1])^{\otimes 2}
  \xto{h}
  d\delta_X^*E\otimes L^{\vee}[1]
  \xto{h}
  d\delta_X^*E.
$$
Next, notice that this totalization is \emph{eventually two-periodic}: for $n\gg 0$, the 3-term cochain complex
\begin{align*}
  \on{Tot} \Bigt{ \dots \xto{h} d\delta_X^*E\otimes L^{\vee} [1] \xto{h} d\delta_X^*E}_{-2n-1} & \to \on{Tot}\Bigt{ \dots \xto{h} d\delta_X^*E\otimes L^{\vee} [1] \xto{h} d\delta_X^*E }_{-2n}\\
   & \to \on{Tot} \Bigt{ \dots \xto{h} d\delta_X^*E\otimes L^{\vee} [1] \xto{h} d\delta_X^*E}_{-2n+1}
\end{align*}
  
identifies with
$$
  \bigoplus_{i\in \bbZ}(d\delta_X^*E)^{2i-1}\otimes L ^{\otimes -i}
  \xto{d+h}
  \bigoplus_{i\in \bbZ}(d\delta_X^*E)^{2i}\otimes L ^{\otimes -i}
  \xto{d+h}
  \bigoplus_{i\in \bbZ}(d\delta_X^*E)^{2i+1}\otimes L ^{\otimes -i}.
$$

\medskip

The resulting two-periodic complex is by definition the object
$$
  \Xi (d\delta^*_XE) \in \MF(X,L,0)_Z.
$$
In particular, we have that (for $n \gg 0$)
$$
  \ccH^{0}\bigt{\Xi(d\delta^*_XE)}
  \simeq
  \ccH^{-2n}(\delta_{X*}\delta_X^*(E))
  =:
  \ul{\Tor}_{2n}^{X\times_SX}(\Delta_X,E),
$$
$$
  \ccH^{1}\bigt{\Xi(d\delta^*_XE)}
  \simeq
  \ccH^{-2n-1}(\delta_{X*}\delta_X^*(E))
  =:
  \ul{\Tor}_{2n+1}^{X\times_SX}(\Delta_X,E).
$$
This immediately implies that 
$$
  [d\delta_X^*E]
  =
  [\ul \Tor^{X\times_SX}_{2n}(\Delta_X,E)]-
  [\ul \Tor^{X\times_SX}_{2n+1}(\Delta_X,E)]
  \in \HK_0^{\bbQ}\bigt{\MF(X,L,0)_Z}
  \simeq
  \uG_0^{\bbQ}(Z)
$$
as claimed.
\end{proof}

\sec{The case of unipotent monodromy}\label{sec:unip-monodromy}

Here we explain the second (and last) part of the proof of Theorem \ref{mainthm:hypersurface}. 
In Section \ref{ssec: cat dimtot}, we use the $\ell$-adic Chern character to define the categorical total dimension and to deduce a categorical version of the generalized Bloch conductor formula. Then, starting from Section \ref{ssec:duality for rl(T)}, we use the unipotence assumption to prove that the categorical version agrees with the classical version.

\ssec{The \texorpdfstring{$\ell$}{l}-adic realization of dg-categories and the categorical total dimension} \label{ssec: cat dimtot}

We now recall the $\ell$-adic realization of dg-categories (defined in \cite{brtv18}) and the $\ell$-adic Chern character (defined in \cite{tv22}). We use the former to define the \emph{categorical total dimension} and the latter to obtain a version of the generalized Bloch conductor conjecture.

\sssec{} 

Following \cite{brtv18}, consider the $\ell$-adic realization $\infty$-functor
\begin{equation}
    \Rl_S : \Mod_{\BU_{\bbQ,S}}(\SH_{S})\rightarrow \Mod_{\Ql{S}(\beta)}(\shvl(S))
\end{equation}
defined in \cite{cd16,ay14}
and consider the composition
\begin{equation}
    \rl_{S}:=\Rl_S\circ \Mv_{\bbQ,S}: \dgCat_A \rightarrow \Mod_{\Ql{S}(\beta)}(\shvl(S)).
\end{equation}
This is the \emph{$\ell$-adic realization of dg-categories}.
It is immediate that $\rl_{S}$ has similar properties to those of $\Mv_{S}$ (as both $-\otimes \HQ$ and $\Rl_{S}$ preserve them).

\sssec{}
As in \cite[Section 3.1]{tv22}, we will choose an isomorphism
$$
  \lim_{(n,p)=1}\mu_n(k)\simeq \widehat{\bbZ}'.
$$
Consequently, we get isomorphisms 
$\Qell \simeq \Qell(n)$, for all $n\in \bbZ$, which trivialize the Tate twists.

\sssec{}
As explained in \cite[\S 2.3]{tv22}, there exists a unique (up to a contractible space of choices) lax-monoidal natural transformation
\begin{equation*}
    \chern_{S}: \HK \rightarrow |\rl_{S}|,
\end{equation*}
called the \emph{non-commutative $\ell$-adic Chern character}. Here, $|-|$ denotes the Dold--Kan construction. We now use this Chern character to obtain an equality of $\ell$-adic rational numbers starting from the dg-functors introduced in Section \ref{sec:K-intersection-thry}.

\sssec{}

Consider the $A$-linear dg-functor
$$
\Perf(S) \to \Coh(X \times_S X) \to \Sing(X \times_S X)
$$ 
determined by $\Delta_X$ (viewed as an object of $\Sing(X \times_S X)$). By abuse of notation, we denote this dg-functor by $\Delta_X$. Now consider its image under ${\Mv_{\bbQ,S}}$ and compose it with the integration map
$$
\int_{X/S}
:
 \Mv_{\bbQ,S}\bigt{\Sing(X\times_S X)} 
\to 
\Mv_{\bbQ,S}\bigt{\MF(S,0)_s},
$$
thus obtaining a morphism
\begin{equation} \label{eqn:motive for Bloch number}
    \BU_{\bbQ,S}
    \simeq 
    \Mv_{\bbQ,S}\bigt{\Perf(S)}
    \xto{{\Mv_{\bbQ,S}}(\Delta_X)} 
    \Mv_{\bbQ,S}\bigt{\Sing(X\times_S X)}
    \xto{\int_{X/S}} 
    \Mv_{\bbQ,S}\bigt{\MF(S,0)_s}.
\end{equation}

\sssec{} 

Recall that, at the level of $\HK_0^{\bbQ}$, the above morphism is the linear map $\bbQ \to \bbQ$ determined by $\Bl(X/S)$.
We now apply $\CR_{\ell}$ to \eqref{eqn:motive for Bloch number} and obtain a $\Ql{S}(\beta)$-linear map
\begin{equation*}
    \Ql{S}(\beta)\simeq \rl_{S}(\Perf(S))
    \xto{\rl_S(\Delta_X)} 
    \rl_{S}(\Sing(X\times_S X)) 
    \xto{\int_{X/S}} 
    \rl_{S}(\MF(S,0)_s)
    \simeq
    (i_S)_*\Ql{s}^{\In}(\beta).
\end{equation*}
This map is determined by an element of 
$$
\pi_0
\Bigt{
\Hom_{\Ql{S}(\beta)}\bigt{\Ql{S}(\beta),
\rl_{S}(\MF(S,0)_s)}}
\simeq 
\Qell,
$$
whose negative we call the \emph{categorical total dimension}.

\begin{cor}\label{cor:cat gen BCC}
In the setting of Theorem \ref{mainthm:hypersurface} (but we no need for the unipotent assumption yet), the $\ell$-adic rational number $\dimtot^{\cat}(X/S)$ belongs to $\bbZ$ and we have
\begin{equation*}
    \Bl(X/S) =-\dimtot^{\cat}(X/S).
\end{equation*}
\end{cor}

\begin{proof}
For the dg-category $\MF(S,0)_s$, the $\ell$-adic Chern character is simply the inclusion $\bbQ \hto \Qell$. Hence, we obtain an equality
$$
\int_{X/S} [\Delta_X]
= - \dimtot^{\cat}(X/S)
$$
of $\ell$-adic rational numbers. Since the left-hand side is an integer, so is the right-hand side.
\end{proof}

\sssec{} 

We conjecture that the above formula coincides with the generalized Bloch conductor formula. We will prove this, under the unipotence assumption, by showing that $\dimtot^{\cat}(X/S) = \dimtot(X/S)$. The rest of Section \ref{sec:unip-monodromy} is devoted to showing that these two total dimensions agree.

\ssec{A duality datum} \label{ssec:duality for rl(T)}

Recall that $\sT$ is dualizable as a left $\sB$-module and that the explicit duality datum was exhibited in Section \ref{ssec: explicit duality datum}.
In general, $\rl_S$ does not preserve dualizability; however, we claim that $\rl_S(\sT)$ is dualizable over $\rl_S(\sB)$. In this section, we construct this duality datum.

\sssec{}

Consider the arrow 
\begin{equation*}
    \rl_{S}(\coev):\rl_{S}(A)\rightarrow \rl_{S}(\T^{\op}\otimes_{\B}\T).
\end{equation*} 
By \cite[Proof of Theorem 5.2.2]{tv22}, this induces an arrow
\begin{equation} \label{eqn:aux1}
    \wt{\coev}_{\rl_S(\sT)}:
    \rl_{S}(A)\rightarrow \rl_{S}(\T^{\op})\otimes_{\rl_{S}(\B)}\rl_{S}(\T).
\end{equation}

\sssec{}

Recall now that $\B$ is not symmetric monoidal. However, the algebra object 
$$
\rl_{S}(\B)
\simeq 
i_*\QellSIbeta
\in \Mod_{\QellSbeta}(\shvl(S))
$$
is commutative. In view of this observation, we see that 
 $\rl_{S}(\T^{\op})\otimes_{\rl_{S}(\B)}\rl_{S}(\T)$ is a $\rl_{S}(\B)$-module and so \eqref{eqn:aux1} induces a map
\begin{equation} \label{eqn:candidate-coev-for-rl}
  \rl_{S}(\B)\rightarrow \rl_{S}(\T^{\op})\otimes_{\rl_{S}(\B)}\rl_{S}(\T),
\end{equation}
which is our candidate coevaluation. We call it   $\coev_{\rl_{S}(\T)}$.

\sssec{}

Let us now turn to the construction of the candidate evaluation for $\rl_S(\sB)$. Start with the evaluation functor $\ev:\T \otimes_A \T^{\op} \rightarrow \B$, which is $\B^{\rev}\otimes_A \B$-linear.
Applying $\rl_S$ to it, we have the map
$$
\rl_{S}(\ev)
:
\rl_S(\sT \otimes_A \sT^{\op}) 
\to 
\rl_S(\sB)
$$
which we can pre-compose with the arrow 
$$
\mu:
\rl_{S}(\T)\otimes_{\rl_{S}(A)}\rl_{S}(\T^{\op})
\to 
\rl_{S}(\T \otimes_A \T^{\op})
$$
given by the lax-monoidal structure on $\rl_{S}$. We obtain an arrow
$$
\wt{\ev}_{\rl_S(\sT)}:
    \rl_{S}(\T)\otimes_{\rl_{S}(A)}\rl_{S}(\T^{\op})\rightarrow \rl_{S}(\B)
$$
which is, by construction, $\rl_{S}(\B^{\rev}) \otimes_{\rl_{S}(A)}\rl_{S}(\B)$-linear. Tensoring up with $\rl_S(\sB)$, we obtain an arrow
$$
    \rl_{S}(\T)\otimes_{\rl_S(\B)}\rl_{S}(\T^{\op})
    \to
    \HH_*(\rl_{S}(\B)/\rl_{S}(A)).
$$

\sssec{}

Since $\rl_{S}(\B)$ is a commutative ring, there is a canonical morphism $\HH_*(\rl_{S}(\B)/\rl_{S}(A))\to \rl_{S}(\B)$. Composing with this, we obtain a map
\begin{equation} \label{eqn:candidate-ev-for-rl}
     \rl_{S}(\T)\otimes_{\rl_S(\B)}\rl_{S}(\T^{\op})\rightarrow \rl_{S}(\B),
\end{equation}
which is our candidate evaluation for $\rl_S(\sT)$ over $\rl_S(\sB)$. We will denote it by $\ev_{\rl_{S}(\T)}$.

\begin{prop}\label{prop:l-adic realization sing is dualizable over B}
With the above notation, the morphisms 
\begin{equation*}
    \coev_{\rl_{S}(\T)}: \rl_{S}(\B)\rightarrow \rl_{S}(\T^{\op})\otimes_{\rl_{S}(\B)}\rl_{S}(\T)
\end{equation*}
\begin{equation*}
    \ev_{\rl_{S}(\T)}: \rl_{S}(\T^{\op})\otimes_{\rl_{S}(\B)}\rl_{S}(\T)\rightarrow \rl_{S}(\B) 
\end{equation*}
appearing in \eqref{eqn:candidate-coev-for-rl} and \eqref{eqn:candidate-ev-for-rl} exhibit $\rl_{S}(\T)$ as a dualizable $\rl_{S}(\B)$-module.
\end{prop}

\begin{proof}
Consider the commutative diagram below, where the unnamed arrows are induced by the lax-monoidal structure of $\rl_S$. To reduce clutter, we write $r$ in place of $\rl_{S}$. 
\begin{small}
\begin{equation*}
    \begin{tikzpicture}[scale=1.5]
    %% upper row %%
    \node (LLu) at (-4.9,1) {$r(\T)$};
    \node (Lu) at (-3.5,1) {$r(\T)\otimes_{r(A)} r(A)$};
    \node (Cu) at (0,1) {$r(\T)\otimes_{r(A)}r(\T^{\op})\otimes_{r(\B)}r(\T)$};
    \node (Ru) at (3.5,1) {$r(\B)\otimes_{r(\B)}r(\T)$};
    \node (RRu) at (4.9,1) {$r(\T)$};
    %% central row %%
    \node (LLc) at (-4.9,-1) {$r(\T\otimes_AA)$};
    \node (Lc) at (-1.8,0) {$r(\sT)\otimes_{r(A)}r(\T^{\op}\otimes_{\B}\T)$};
    \node (Rc) at (1.8,0) {$r(\sT \otimes_{A}\T^{\op})\otimes_{r(\B)}r(\T)$};
    \node (RRc) at (4.9,-1) {$r(\B\otimes_{\B}\T)$};
    %% lower row %%
    \node (Cd) at (0,-1) {$r(\T\otimes_A\T^{\op}\otimes_{\B}\T)$};
    %% horizontal arrows %%
    \draw[->] (LLu) to node[above] {$\sim$} (Lu);
    \draw[->] (Lu) to node[above] {$\id \otimes \wt{\coev}_{r(\T)}$} (Cu);
    \draw[->] (Cu) to node[above] {$\wt{\ev}_{r(\T)}\otimes \id$} (Ru);
    \draw[->] (Ru) to node[above] {$\sim$} (RRu);
    %
    %% diagonal arrows upper row -> central row %%
    \draw[->] (LLu) to node[left] {$\sim$} (LLc);
    \draw[->] (Lu) to node[left] {$\id \otimes r(\coev)\;\;$} (Lc);
    \draw[->] (Cu) to node[left] {${}$} (Lc);
    \draw[->] (Cu) to node[left] {${}$} (Rc);
    \draw[->] (Rc) to node[right] {$\;\; r(\ev)\otimes \id$} (Ru);
    \draw[->] (RRc) to node[right] {$\sim$} (RRu);
    %% diagonal arrows central row -> lower row %%
    \draw[->] (LLc) to node[below] {$r(\id \otimes \coev) \;\;\;\;$} (Cd);
    \draw[->] (Lc) to node[left] {${}$} (Cd);
    \draw[->] (Rc) to node[left] {${}$} (Cd);
    \draw[->] (Cd) to node[below] {$\;\;\;\; r(\ev \otimes \id)$} (RRc);
    \end{tikzpicture}
\end{equation*}
\end{small}
A diagram chase shows that
$$
(\ev_{\rl_{S}(\T)}\otimes \id)\circ (\id\otimes \coev_{\rl_{S}(\T)})
\simeq 
\rl_S\bigt{(\ev\otimes \id) \circ (\id \otimes \coev)}
\simeq
\id_{\rl_{S}(\T)}.
$$
The other identity
$$
(\id\otimes \ev_{\rl_{S}(\T)})\circ (\coev_{\rl_{S}(\T)}\otimes \id)
\simeq 
\id_{\rl_{S}(\T)}
$$
is proven similarly.
\end{proof}
\begin{rmk}
    Notice that $\rl_S(\T)\simeq \rl_S(\T^{\op})$, so that the above proposition implies that $\rl_S(\T)$ is self-dual.
\end{rmk}

\begin{rmk}
    Notice that in Proposition \ref{prop:l-adic realization sing is dualizable over B} only the hypothesis of unipotent action is needed. 
    In particular, the proposition is valid regardless of the fact that $X$ embeds as an hypersurface in a smooth $S$-scheme.
\end{rmk}

\ssec{Integration and evaluation}

The goal of this section is to relate the integration map
$$
\int_{X/S}:
\rell_S \bigt{\Sing(X \times_S X)}
\to 
\rell_S \bigt{ \MF(S,0)_s}.
$$
with the evaluation $\ev_{\rl_S(T)}$. Namely, we will prove the following compatibility.

\begin{prop} \label{prop:int-vs-ev}
The diagram 
\begin{equation} 
\nonumber
\begin{tikzpicture}[scale=1.5]
\node (00) at (0,0) {$ \rl_S(\B)  $};
\node (10) at (3,0) {$\rl_S(\MF(S,0)_s)$ };
\node (01) at (0,1) {$ \rl_S (\sT)\otimes_{\rl_S(\B)}\rl_S(\sT)$};
\node (11) at (3,1) {$ \rl_S(\Sing(X \times_S X)) $};
\path[->,font=\scriptsize,>=angle 90]
(00.east) edge node[above] {$\int_{s/S}$}  (10.west); 
\path[->,font=\scriptsize,>=angle 90]
(01.east) edge node[above] {$ \simeq $} (11.west); 
%%%VERT
\path[->,font=\scriptsize,>=angle 90]
(01.south) edge node[right] {$\ev_{\rl_S(\sT)} $} (00.north); 
\path[->,font=\scriptsize,>=angle 90]
(11.south) edge node[right] {$ \int_{X/S} $} (10.north);
\end{tikzpicture}
\end{equation}
is commutative.
\end{prop}

\sssec{}

Before starting the proof, let us introduce the following notation.
Let 
$$
i_\T:\rl_S(\sT)\otimes_{\rl_S(A)}\rl_S(\sT)\to \rl_S\bigt{\Sing(X\times_SX)}
$$ 
be the morphism obtained by the lax monoidal structure on $\rl_S$, the canonical morphism $\T\otimes_A\sT^{\op}\to \T^{\op}\otimes_\B \sT$ and $\ffF: \T^{\op}\otimes_\B \sT \to \Sing(X \times_S X)$.

\begin{lem}\label{lem: left square tech-prop}
The diagram 
\begin{equation}  \label{diag:pre-big-comm-diagram-rl (GR)}
  \begin{tikzpicture}[scale=1.5]
    \node (Lu) at (-4,1.2) {$\rl_S(\sT) \otimes_{\rl_S(A)} \rl_S(\sT)$};
    \node (Ld) at (-4,0) {$\rl_S(\sB)$};
    \node (Cu) at (0,1.2) {$\rl_S\bigt{\Sing(X\times_S X)}$};
    \node (Cd) at (0,0) {$ \rl_S\bigt{\MF(S,0)_s}$};
    %%% vertical arrows %%%
    \draw[->] (Lu) to node[right] { $\wt{\ev}_{\rl_S(\sT)}$ } (Ld);
    \draw[->] (Cu) to node[right] { $\int_{X/S}$ } (Cd);
    %%% horiz arrows %%%
    \draw[->] (Lu) to node[above]{$i_\T$} (Cu);
    \draw[->] (Ld) to node[above]{$\int_{s/S}$} (Cd);
  \end{tikzpicture}
  \end{equation}
is commutative.
\end{lem}

\begin{proof}
We proceed in steps.
\sssec*{Step 1}
Both $\rl_S(\sT)$ and $\rl_S(\sB)$ are dualizable over $\rl_S(A)$.
An easy computation shows that 
$$
\wt{\ev}_{\rl_S(\sT)}:\rl_S(\sT)\otimes_{\rl_S(A)}\rl_S(\sT)\to \rl_S(\sB)
$$
identifies via duality to a morphism
\begin{equation}\label{eqn: act r(B) on r(T)}
    \rl_S(\sB)\otimes_{\rl_S(\sA)}\rl_S(\sT)
    \to
    \rl_S(\sT).
\end{equation}
This is the action of $\rl_S(\sB)$ on $\rl_S(\sT)$.
By \cite[Proposition 4.27, Theorem 4.39]{brtv18}, \eqref{eqn: act r(B) on r(T)} is equivalent to a morphism
\begin{equation*}
    (i_S)_*(p_s)_*\bigt{\Ql{X_s}^{\In}(\beta)\otimes_{\Ql{X_s}(\beta)}\scrV_{X/S}^{\In}[-1]}
    \to
    (i_S)_*(p_s)_*\scrV_{X/S}^{\In}[-1],
\end{equation*}
which is the image along $(i_S)_*(p_s)_*$ of the morphism 
\begin{equation}\label{eqn: morphism a}
    a:
    \Ql{X_s}^{\In}(\beta)\otimes_{\Ql{X_s}(\beta)}\scrV_{X/S}^{\In}[-1]
    \to
    \scrV_{X/S}^{\In}[-1]
\end{equation}
induced by the action of $\Ql{X_s}^{\In}(\beta)$ on $\scrV_{X/S}^{\In}[-1]$.

\sssec*{Step 2}
Thanks to the equivalences
$$
  \rl_S(\sB)
  \simeq
  \rl_S\bigt{\MF(S,0)_s}
  \simeq
  (i_S)_*\QellI(\beta),
$$
the morphism $\int_{X/S}\circ\, i_{\T}:\rl_S(\sT)\otimes_{\rl_S(A)}\rl_S(\sT)\to \rl_S(\sB)$ identifies, via duality, to a morphism
$$
  \rl_S(\sB)\otimes_{\rl_S(\sA)}\rl_S(\sT)
  \to
  \rl_S(\sT).
$$
Invoking once again \cite[Proposition 4.27, Theorem 4.39]{brtv18}, this identifies with a morphism
\begin{equation}\label{eqn: int circ i_T via duality and equivalences}
  (i_S)_*(p_s)_*\bigt{\Ql{X_s}^{\In}(\beta)\otimes_{\Ql{X_s}(\beta)}\scrV_{X/S}^{\In}[-1]}
    \to
    (i_S)_*(p_s)_*\scrV_{X/S}^{\In}[-1].
\end{equation}
\sssec*{Step 3}

Next, we describe the latter morphism explicitly.
The $\Perf(P)$-linear dg-category $\Sing(X_s)$ carries an action of $\MF(P,L,0)$. 
This follows immediately from the isomorphism
$$
  X\simeq P\times_{\tau,L,0}P.
$$
In particular, we find that $\rl_P\bigt{\MF(P,L,0)}$ acts on $\rl_P\bigt{\Sing(X_s)}$.
Let $k:Z\to P$ denote the closed embedding of $Z$ inside $P$ and let us identify $\scrV_{X/S}$ with an $\ell$-adic sheaf on $Z$.
Then, since (\cite[Theorem 4.39]{brtv18})
$$
  \rl_P\bigt{\Sing(X_s)}
  \simeq
  k_* \scrV_{X/S}^{\In}[-1],
$$
this action of $\rl_P\bigt{\MF(P,L,0)}$ factors through the canonical morphism of algebras
$$
  \rl_P\bigt{\MF(P,L,0)}
  \to
  k_*k^* \rl_P\bigt{\MF(P,L,0)}.
$$
However, just as in Proposition \ref{prop: MF(X,C,0)_Z decomposes as a direct sum in motives}, we find that 
$$
  k_*k^* \rl_P\bigt{\MF(P,L,0)}
  \simeq
  k_*\Ql{Z}^{\In}(\beta).
$$
Then \eqref{eqn: int circ i_T via duality and equivalences} identifies with the image along $(i_S)_*(Z\to s)_*$ of this action of $\Ql{Z}^{\In}(\beta)$ on $\scrV_{X/S}^{\In}[-1]$. 
We introduce the following notation for future reference:
\begin{equation}\label{eqn: morphism a'}
    a':
    \Ql{Z}^{\In}(\beta)\otimes_{\Ql{Z}(\beta)}\scrV_{X/S}^{\In}[-1]
    \to
    \scrV_{X/S}^{\In}[-1].
\end{equation}

\sssec*{Step 4} 

Thanks to the steps above, in order to prove that $\int_{s/S} \circ \eps_{\sT}$ and $\int_{X/S}\circ i_{\T}$ are homotopic, it suffices to show that $a$ and $a'$ are so, which is a local statement.
Hence, from now we assume that
$$
X
\simeq
S\times_{\bbA^1_S}P
$$
is the zero locus of a function $f: P\to \bbA^1_S$.

\sssec*{Step 5}

Using $P_s := P \times_S s$, we obtain the following alternative presentations of $X_s$ as a fiber product:
$$
X_s
\simeq 
s \times_{\bbA^1_s}P_s
$$
$$
X_s
\simeq
S \times_{\bbA^1_S} P_s.
$$
These two expressions make it evident that $X_s$ is endowed with actions of the derived groups
$$
G^{\on{comm}}=s\times_{\bbA^1_s}s
$$
$$
G_S^{\on{comm}}=S\times_{\bbA^1_S} S.
$$
Since $P_s$ is a regular scheme, we obtain actions of $\Sing(G^{\on{comm}})$ and $\Sing(G_S^{\on{comm}})$ on $\Sing(X_s)$. 
By construction, the former action induces the morphism \eqref{eqn: morphism a'}, while the morphism \eqref{eqn: morphism a} is induced by the usual action of $\Sing(G)$ on $\Sing(X_s)$.

To conclude, it suffices to notice that these action coincide at the level of $\ell$-adic sheaves. The reason they do is that they are both induced by the action of $\Sing\bigt{G_S^{\on{comm}}}$. Namely, consider the commutative diagram
\begin{equation*}
    \begin{tikzpicture}[scale=1.5]
        \node (Lu) at (0,1) {$\rl_S\bigt{\Sing(G_S^{\on{comm}})}$};
        \node (Ru) at (3,1) {$\rl_S\bigt{\Sing(G^{\on{comm}})}$};
        \node (Ld) at (0,0) {$\rl_S\bigt{\Sing(G)}$};
        \node (Rd) at (3,0) {$(i_S)_*(i_S)^*\Qell^{\In}(\beta).$};

        \draw[->] (Lu) to node[above] {${}$} (Ru);
        \draw[->] (Lu) to node[above] {${}$} (Ld);
        \draw[->] (Ru) to node[right] {$\simeq$} (Rd);
        \draw[->] (Ld) to node[above] {$\simeq$} (Rd);
    \end{tikzpicture}
\end{equation*}
Under the equivalence $\rl_S\bigt{\Sing(G_S^{\on{comm}})}\simeq \Ql{S}^{\In}(\beta)$, both arrows $\rl_S\bigt{\Sing(G_S^{\on{comm}})}\to (i_S)_*(i_S)^*\Ql{S}^{\In}(\beta)$ identify with the canonical arrow $\Ql{S}^{\In}(\beta)\to (i_S)_*(i_S)^*\Ql{S}^{\In}(\beta)$.
\end{proof}

\sssec{} 

To deduce Proposition \ref{prop:int-vs-ev} from the above lemma, it suffices to observe that the arrows of
\eqref{diag:pre-big-comm-diagram-rl (GR)}
are all $\rl_S(\sB)$-linear.

\ssec{Conclusion of the proof of Theorem \ref{mainthm:hypersurface}} \label{ssec:proof of Thm B}

Here we show that, in the setting of Theorem \ref{mainthm:hypersurface}, the categorical total dimension we defined in Section \ref{ssec: cat dimtot} coincides with the usual total dimension.
In view of Corollary \ref{cor:cat gen BCC}, this proves Theorem \ref{mainthm:hypersurface}.

\begin{thm}\label{main theorem 3}

Let $p: X \to S$ as in Theorem \ref{mainthm:hypersurface}.
For $f:X\to X$ an $S$-endomorphism, denote by $\Gamma_f=(\id,f)_*\CO_X$ the structure sheaf of the graph of $f$.
Then
$$
[[\Delta_X,\Gamma_f]]_S
=
\Tr_{\Qell}
\bigt{(f_s)_*,\uH^*_{\et}(X_s,\Qell)}
-
\Tr_{\Qell}
\bigt{(f_{\bareta})_*,\uH^*_{\et}(X_{\bareta},\Qell)}.
$$
In particular, for $f=\id$, we obtain
$$
\Bl(X/S)=\chi(X_s;\Qell)-\chi(X_{\bareta};\Qell)
=
-\dimtot \bigt{\uH^*_{\et}(X_s,\Phi_{X/S})}.
$$
\end{thm}

The remaining part of Section \ref{ssec:proof of Thm B} is devoted to the proof of this theorem.

\sssec{}

Proposition \ref{prop:int-vs-ev} implies that the diagram
\begin{equation} 
\nonumber
\begin{tikzpicture}[scale=1.5]
\node (00) at (.5,0) {$\uH^0_{\et}\bigt{S,\rl_S(\B)}$} ;
\node (10) at (4,0) {$\uH^0_{\et}\Bigt{S,\rl_S\bigt{\MF(S,0)_s}}$ };
\node (01) at (.5,1) {$\uH^0_{\et}\bigt{S,\rl_S(\T^{\op} \otimes_{\B} \T)}$};
\node (m31) at (-3,1) {$\uH^0_{\et}\bigt{S,\rl_S(\T^{\op})\otimes_{\rl_S(\B)}\rl_S(\T)}$};
\node (11) at (4,1){$ \uH^0_{\et}\Bigt{S,\rl_S\bigt{\Sing(X \times_S X)}} $};
\path[->,font=\scriptsize,>=angle 90]
(00.east) edge node[above] {$\simeq$}  (10.west); 
\path[->,font=\scriptsize,>=angle 90]
(01.east) edge node[above] {$ $} (11.west); 
\path[->,font=\scriptsize,>=angle 90]
(m31.east) edge node[above] {$ \simeq $} (01.west); 
%%%VERT
\path[->,font=\scriptsize,>=angle 90]
(01.south) edge node[left] {$ {} $} (00.north);
\path[->,font=\scriptsize,>=angle 90]
(11.south) edge node[right] {$ \int_{X/S}  $} (10.north);
\end{tikzpicture}
\end{equation}
is commutative, where the composition 
$$ \uH^0_{\et}\bigt{S,\rl_S(\T^{\op})\otimes_{\rl_S(\B)}\rl_S(\T)} \to \uH^0_{\et}\bigt{S,\rl_S(\B)}
$$
is, by definition, the map induced by $\ev_{\rl_S(\T)}$.

\sssec{}

By Remark \ref{rem: endomorphism induced by push and pull on T}, the map $(\id \otimes (f_s)_*)\circ \coev_{\rl_S(\T)}$
corresponds to the cohomology class
$$
\chern_{S}([\Gamma_f]) \in \uH^0_{\et}\Bigt{S,\rl_S\bigt{\Sing(X \times_S X)}}. 
$$
In particular, we find that 
\begin{equation}\label{eqn: int Gamma_f = trace r(T)}
\int_{X/S} \Bigt{ \chern_{S}([\Gamma_f]) }= \Tr_{(i_S)_*\Qell(\beta)}\bigt{(f_s)_*,\rl_S(\T)},
\end{equation}
as elements of $\uH^0_{\et}\Bigt{S,\rl_S\bigt{\MF(S,0)_s}}
\simeq \Qell$.

\sssec{}

The main theorem in \cite{brtv18} yields
$$
  \rl_S(\T)\simeq (i_S)_*\uH^*_{\et} (X_s,\scrV_{X/S})^{\In}[-1],
$$
that is, the $\ell$-adic realization of $\T$ recovers inertia-invariant vanishing cycles.
Moreover, by \cite[Lemma 5.2.5]{tv22}, taking fixed points with respect to $\In$ behaves as a symmetric monoidal functor when applied to complexes with unipotent action. 
Thus,
\begin{equation}\label{eqn: trace r(T)= trace Phi}
\Tr_{(i_S)_*\Qell(\beta)}\bigt{(f_s)_*,\rl_S(\T)} =  \Tr_{\Qell}\bigt{(f_s)_*,\uH^*_{\et}(X_s,\Phi_{X/S})[-1]}.                             
\end{equation}

\sssec{}

On the other hand,
\begin{equation}\label{eqn: int chern Gamma_f = chern int Gamma_f}
\int_{X/S}\Bigt{\chern_{S}([\Gamma_f])}= \chern_{S}\Bigt{\int_{X/S}[\Gamma_f]}
\end{equation}
and the map
$$
  \chern_{S}: \bbQ
  \simeq 
  \HK^{\bbQ}_0\bigt{\MF(S,0)_s}
  \to 
  \uH^0_{\et}\Bigt{S,\rl_S\bigt{\MF(S,0)_s}}
  \simeq 
  \Qell
$$
is just the inclusion of the rational numbers into the $\ell$-adic rational numbers.
We will implicitly compose with such a map in the computation that follows.

\sssec{}

Summarizing all the steps above, we have obtained the following chain of equalities:
\begin{align*}
[[\Delta_X,\Gamma_f]]_S & =  \int_{X/S}[\Gamma_f] && \text{Theorem \ref{thm: KS loc int prod = HK_0 int with the diagonal}}\\
                                     & =  \int_{X/S} \Bigt{\chern_{S}([\Gamma_f])} && \text{\eqref{eqn: int chern Gamma_f = chern int Gamma_f}}\\
                                     & = \Tr_{(i_S)_*\Qell(\beta)}(\rl_S((f_s)_*);\rl_S(\T)) && \text{\eqref{eqn: int Gamma_f = trace r(T)}}\\
                                     & =  \Tr_{\Qell}\bigt{(f_s)_*,\uH^*_{\et}(X_s,\Phi_{X/S}[-1])}. && \text{\eqref{eqn: trace r(T)= trace Phi}}
\end{align*}

\sec{The pure characteristic case}\label{sec:pure-char}

The goal of this section is to prove Theorem \ref{mainthm:pure-char}. The strategy is very similar to the one employed for Theorem \ref{mainthm:hypersurface}. The only difference is that the categorification of the Kato--Saito localized intersection product can be done canonically (that is, we do not need the presentation of $X$ as a hypersurface).

\ssec{Intersection with the diagonal in the pure characteristic case}

\sssec{}

Assume that $S$ is of pure characteristic and let $p:X\to S$ be as in the statement of the generalized Bloch conductor formula.
The arrow $S\to s$ (a feature of the pure characteristic case) lets us consider the following diagram, where both squares are Cartesian:
\begin{equation*}
        \begin{tikzpicture}[scale=1.5]
            \node (02) at (0,2) {$K(X,0)$};
            \node (12) at (2,2) {$X$};
            \node (01) at (0,1) {$X\times_SX$};
            \node (11) at (2,1) {$X\times_sX$};
            \node (00) at (0,0) {$S$};
            \node (10) at (2,0) {$S\times_sS.$};
            %% horiz maps %%
            \draw[->] (02) to node[above] {$$} (12);
            \draw[->] (01) to node[above] {$$} (11);
            \draw[->] (00) to node[above] {$\delta_S$} (10);
            %% vertical maps %%
            \draw[->] (02) to node[left] {$d\delta_X$} (01);
            \draw[->] (12) to node[above] {$$} (11);
            \draw[->] (01) to node[above] {$$} (00);
            \draw[->] (11) to node[right] {$$} (10);
        \end{tikzpicture}
    \end{equation*}
Similarly to the hypersurface case, we have a factorization of the diagonal
$$
  \delta_X:
  X
  \xto{\iota}
  K(X,0)
  \xto{d\delta}
  X\times_SX,
$$
so that $d\delta_X$ is a \emph{derived} enhancement of the diagonal.

\sssec{}

Notice that, since $X$ is smooth over $s$, the diagonal $X\to X\times_sX$ is quasi-smooth. 
Therefore, $d\delta_X$ is a quasi-smooth closed embedding and so it induces a dg-functor
$$
    d\delta_X^*:
    \Sing(X\times_SX)
    \to
    \Sing \bigt{K(X,0)}.
$$
The proof of Lemma \ref{lem: pullback along derived pullback factors through MF with support on the singular locus} works in this case too, \emph{mutatis mutandis}. Hence, just as in the hypersurface case, this dg-functor factors through
\begin{equation} \label{eqn:d-delta-in-pure char}
    d\delta_X^*:
    \Sing(X\times_SX)
    \to
    \Sing \bigt{K(X,0)}_Z.
\end{equation}
Recall that $Z$ denotes the singular locus of $X \times_S X$.

\sssec{}

The next step is simpler than the corresponding one for Theorem \ref{mainthm:hypersurface}. Namely, the push-forward along $K(X,0) \to K(S,0)$ induces a dg-functor
$$
  \Sing \bigt{K(X,0)}_Z
  \to
  \Sing \bigt{K(S,0)}_s.
$$
Pre-composing with \eqref{eqn:d-delta-in-pure char}, we obtain a dg-functor
$$
  \int_{X/S}^{\on{dg}}:
  \Sing(X\times_SX)
  \to
  \Sing \bigt{K(S,0)}_s.
$$

\begin{prop}
Let
\begin{equation*}
    [\Delta_X,-]: \uG_0(X\times_SX) \to \HK_0\bigt{\Sing(X\times_S X)} \rightarrow \HK_0\bigt{\MF(X,0)_Z}\simeq \uG_0(Z)
\end{equation*}   
denote the morphism induced by $d\delta_X^*$. Then
\begin{equation*}
    [\Delta_X,-]=[[X,-]]_{X\times_S X},  
\end{equation*}
where the right-hand side denotes the localized intersection product defined by Kato--Saito in \cite[Definition 5.1.5]{ks04}).
\end{prop}

\begin{proof}
    The proof of Theorem \ref{thm: KS loc int prod = HK_0 int with the diagonal} works in this case as well. 
    Since the line bundle is trivial in this case, we don't need to invoke \cite[\S3]{pi22}: we invoke \cite[\S3]{brtv18} instead, so that we don't need to tensor with $\bbQ$.
\end{proof}

\begin{cor}\label{cor: nc interpretation Bloch number pure char}
The dg-functor $\int_{X/S}^{\on{dg}}: \Sing(X\times_SX) \to \Sing \bigt{K(S,0)}_s$
induces the Kato-Saito localized intersection product
$$
  \uG_0(X\times_SX) 
  \to 
  \uG_0(s)
  \simeq
  \bbZ
$$
$$
  [E]
  \mapsto
  [[\Delta_X,E]]_S.
$$
In particular, we have that $\int_{X/S}[\Delta_X]=\Bl(X/S)$.
\end{cor}

\ssec{Proof of Theorem \ref{mainthm:pure-char}}

\sssec{}

Recall that the setting of this theorem requires $S$ to have pure characteristic. 
The main extra feature of this case is that the closed embedding $s \hto S$ admits a retraction $S \tto s$. As a consequence, the convolution monoidal structure on $\sB$ is \emph{symmetric}. In fact, by Koszul duality, 
$$
\sB \simeq  \Perf\bigt{k[u,u^{-1}]},
$$ 
where $u$ is a free variable of degree $2$. 
Under this equivalence, the symmetric monoidal on $\sB$ corresponds to the canonical symmetric monoidal structure on $\Perf\bigt{k[u,u^{-1}]}$.

\begin{rem}
The above equivalence can alternatively be written as 
$$
\sB \simeq \MF(s,0),
$$
again with the usual symmetric monoidal structure on the right-hand side.
\end{rem}

\sssec{}

Since $\sB$ is symmetric monoidal, the structure morphism $\sB \to \HH_*(\sB/A)$ admits a retraction
$$
  \on{m}:
  \HH_*(\sB/A)
  \to
  \sB.
$$
We will make essential use of this morphism in the computations below.

\sssec{}

Consider the canonical morphism
$$
\Sing(X \times_S X)
\simeq
\sT^{\op}
\otimes_{\sB}
\sT
\to \HH_*(\sB/A)
$$
and the composition
\begin{equation} \label{eqn:from Sing(XSX) to B in geom case}
\Sing(X \times_S X)
\simeq
\sT^{\op}
\otimes_{\sB}
\sT
\to \HH_*(\sB/A)
\xto{\on{m}}
\sB.
\end{equation}
Taking $\HK_0$, we obtain arrows
\begin{equation} \label{eqn:pure-char--degree map}
\uG_0(X \times_S X)
\to 
\HK_0 \bigt{\Sing(X \times_S X)}
\to 
\HK_0(\sB) 
\simeq 
\bbZ. 
\end{equation}
We now show that map computes the difference of the stable Tors.
\begin{prop} \label{prop:stable-Tors-pure-char}
The above map \eqref{eqn:pure-char--degree map} sends $[E] \in \uG_0(X \times_S X)$ to 
$$
\deg 
\big[
 \on{\ul{sTor}}_{0}^{X \times_S X}(\Delta_X, E)
\big]
-
\deg 
\big[
 \on{\ul{sTor}}_{1}^{X \times_S X}(\Delta_X, E)
\big].
$$
\end{prop}

The proof is obtained by looking at the two circuits of the diagram below.

\begin{lem}
The diagram
\begin{equation} 
\begin{tikzpicture}[scale=1.5]
\node (M) at (0,1) {$\hspace{2.5cm}\sT^{\op} \otimes_{\sB} \sT \simeq \Sing(X \times_S X) $ };
\node (IM) at (0,0) {$\HH_*(\sB/A) $};
\node (N) at (3.5,1) {$ \MF(S,0)_s$};
\node (IN) at (3.5,0) {$ \sB \simeq \MF(s,0)$};
\path[ ->,font=\scriptsize,>=angle 90]
(M.south) edge node[right] { $\ev^{\HH}_{\sT/\sB}$ } (IM.north);
\path[ ->,font=\scriptsize,>=angle 90]
(N.south) edge node[right] { $ $ } (IN.north);
%%%horiz maps below
\path[->,font=\scriptsize,>=angle 90]
(M.east) edge node[above]{$\int^{\on{dg}}_{X/S} $} (N.west);
\path[->,font=\scriptsize,>=angle 90]
(IM.east) edge node[above]{$\on{m}$} (IN.west);
\end{tikzpicture}
\end{equation}
is commutative. Here, the right vertical map is induced by the push-forward along $K(S,0) \to K(s,0)$.
\end{lem}

\begin{proof}
Since $\sB$ is symmetric monoidal, the dg-categories $\T^{\op}\otimes_{\sB}\sT$ and $\HH_*(\sB/A)$ admit natural actions of $\sB$. It is clear that the lower circuit of the square is clearly $\sB$-linear. One checks that the upper circuit is $\sB$-linear too. Hence, we are comparing two objects of the dg-category
$$
\Hom_{\sB} 
\bigt{
\T^{\op}\otimes_{\sB}\sT, \sB
}.
$$
Recalling that 
$$ 
\T^{\op}\otimes_{\sB}\sT \simeq 
(\T \otimes_{A}\sT^{\op})
\usotimes_{\sB^{\env}} \sB,
$$
we find a canonical equivalence
$$
\Hom_{\sB} 
\bigt{
\T^{\op}\otimes_{\sB}\sT, \sB
}
\simeq 
\Hom_{\sB^{\env}} 
\bigt{
\T \otimes_A \sT^{\op}, \sB
}
$$
induced by the restriction along the dg-functor $\T^{\op}\otimes_{A}\sT \to \T^{\op}\otimes_{\sB}\sT$. Then the assertion is a diagram chase, using the definitions and the fiber squares
\begin{equation}  \label{diag:big-comm-diagram}
\begin{tikzpicture}[scale=1.5]
\node (L) at (-2,1) {$K(s,0)$ };
\node (IL) at (-2,0) {$K(S,0)$};

\node (M) at (0,1) {$K(X_s,0)$  };
\node (IM) at (0,0) {$K(X,0)$};
\node (N) at (2,1) {$X_s \times_s X_s $};
\node (IN) at (2,0) {$X \times_S X$.};
\path[ ->,font=\scriptsize,>=angle 90]
(L.south) edge node[right] { $ $ } (IL.north);
\path[ ->,font=\scriptsize,>=angle 90]
(M.south) edge node[right] { $i $ } (IM.north);
\path[ right hook ->,font=\scriptsize,>=angle 90]
(N.south) edge node[right] { $j $ } (IN.north);
%%%horiz maps below
\path[->,font=\scriptsize,>=angle 90]
(M.east) edge node[above]{$ $} (N.west);
\path[->,font=\scriptsize,>=angle 90]
(IM.east) edge node[above]{$d \delta$} (IN.west);
\path[<-,font=\scriptsize,>=angle 90]
(L.east) edge node[above]{$ $} (M.west);
\path[<-,font=\scriptsize,>=angle 90]
(IL.east) edge node[above]{$ $} (IM.west);
\end{tikzpicture}
\end{equation}
To see that the right square is cartesian, observe that $(X \times_S X) \times_S s \simeq X_s \times_s X_s$ and then use the canonical isomorphism $K(W,0) \times_W W' \simeq K(W',0)$ valid for any map $W' \to W$.
\end{proof}
\sssec{}
We then obtain a commutative diagram
\begin{equation}\label{diag: key comm diagnostics in pure char} 
\begin{tikzpicture}[scale=1.5]
\node (L) at (-3.5,1) {$\rl_S(\sT^{\op}) \otimes_{\rl_S(\sB)} \rl_S(\sT)$ };
\node (IL) at (-3.5,0) {$\rl_S(\sB) $};

\node (M) at (0,1) {$\rl_S\bigt{\Sing(X \times_S X)} $ };
\node (IM) at (0,0) {$\rl_S\bigt{\HH_*(\sB/A)} $};
\node (N) at (3.5,1) {$ \rl_S(\MF(S,0)_s)$};
\node (IN) at (3.5,0) {$ \rl_S(\sB)$};

\path[ ->,font=\scriptsize,>=angle 90]
(L.south) edge node[right] { $\ev_{\rl_S(\sT)}$ } (IL.north);
\path[ ->,font=\scriptsize,>=angle 90]
(L.east) edge node[above] { $\simeq$ } (M.west);
\path[ ->,font=\scriptsize,>=angle 90]
(IL.east) edge node[right] { ${}$ } (IM.west);

\path[ ->,font=\scriptsize,>=angle 90]
(M.south) edge node[right] { $\rl_S(\ev^{\HH}_{\sT/\sB})$ } (IM.north);
\path[ ->,font=\scriptsize,>=angle 90]
(N.south) edge node[right] { $ \simeq $ } (IN.north);
%%%horiz maps below
\path[->,font=\scriptsize,>=angle 90]
(M.east) edge node[above]{$\int_{X/S} $} (N.west);
\path[->,font=\scriptsize,>=angle 90]
(IM.east) edge node[above]{$\rl_S(\on{m})$} (IN.west);
\end{tikzpicture}
\end{equation}
and we conclude as follows:
\begin{align*}
    \Bl(X/S) & = \int_{X/S}[\Delta_X] && \text{Cor.  \ref{cor: nc interpretation Bloch number pure char}}\\
             & = \ev_{\rl_S(\T)}\bigt{ \chern (\Delta_X)} && \text{\eqref{diag: key comm diagnostics in pure char}}  \\
             & = \Tr_{\rl_S(\sB)}\bigt{\id:\rl_S(\T)} && \text{Prop. \ref{prop:l-adic realization sing is dualizable over B}} \\
             & = - \Tr_{\Qell}\bigt{\id: \uH^*_{\et}(X_s,\Phi_{X/S})} && \text{\cite[Theorem 4.39]{brtv18}}\\
             & = -\dimtot \bigt{\uH^*_{\et}(X_s,\Phi_{X/S})}.
\end{align*}
Notice that we have implicitly used that the $\ell$-adic non commutative Chern character
$$
  \chern:
  \HK_0(\B)
  \to
  \uH^0_{\et}\bigt{S,\rl_S(\B)}
$$
identifies with the inclusion $\bbZ \subseteq \Qell$.
\begin{rmk}
    Notice that the proof of Theorem \ref{mainthm:pure-char} works in the same situation as in Theorem \ref{main theorem 3}.
\end{rmk}

\end{document}